\documentclass[twoside]{amsart}

\numberwithin{equation}{section}
\usepackage{amsmath}
\usepackage{amssymb,amsfonts,amsthm,latexsym}
\usepackage{graphicx,color,enumerate}
\usepackage[all]{xy}

\newtheorem{theorem}{Theorem}[section]
\newtheorem{lemma}[theorem]{Lemma}
\newtheorem{proposition}[theorem]{Proposition}
\newtheorem{corollary}[theorem]{Corollary}

\theoremstyle{definition}

\newtheorem{example}[theorem]{Example}
\theoremstyle{remark}
\newtheorem{remark}[theorem]{Remark}

\numberwithin{equation}{section}

\newcommand{\C}{\mathbb{C}}
\newcommand{\Sp}{\mathbb{S}}
\newcommand{\Q}{\mathbb{Q}}
\newcommand{\R}{\mathbb{R}}

\newcommand{\Z}{\mathbb{Z}}
\newcommand{\N}{\mathbb{N}}

\newcommand{\U}{\operatorname{U}}
\newcommand{\SU}{\operatorname{SU}}
\newcommand{\id}{\operatorname{id}}

\newcommand{\Kern}{\operatorname{ker}}

\newcommand{\diag}{\operatorname{diag}}
\newcommand{\co}{\colon\thinspace}

\newcommand{\bs}{\boldsymbol}
\newcommand{\cc} [1] {\overline {{#1}}}

\newcommand{\Hilb}{\operatorname{Hilb}}
\newcommand{\on}{\mathit{on}}
\newcommand{\off}{\mathit{off}}
\newcommand{\calA}{\mathcal{A}}
\newcommand{\calC}{\mathcal{C}}

\begin{document}

\title{An impossibility theorem for linear symplectic circle quotients}

\author{Hans-Christian Herbig}
\address{Charles University in Prague, Faculty of Mathematics and Physics,
Sokolovsk\'{a} 83, 186 75 Praha 8, Czech Republic}
\email{herbig@imf.au.dk}

\author{Christopher Seaton}
\address{Department of Mathematics and Computer Science,
Rhodes College, 2000 N. Parkway, Memphis, TN 38112}
\email{seatonc@rhodes.edu}

\thanks{The research was supported by the \emph{Centre for the Quantum Geometry of Moduli Spaces},
which is funded by the Danish National Research Foundation. In addition, HCH was supported by the
Austrian Ministry of Science and Research BMWF (Start-Prize Y377) and by the grant GA CR P201/12/G028. CS was supported by
a Rhodes College Faculty Development Grant as well as the E.C. Ellett Professorship in Mathematics.}

\begin{abstract}
We prove that when $d>2$, a $d$-dimensional symplectic quotient at the zero level of a unitary circle
representation $V$ such that $V^{\Sp^1}=\{0\}$ cannot be $\Z$-graded regularly symplectomorphic
to the quotient of a unitary representations of a finite group.
\end{abstract}

\maketitle

\tableofcontents


\section{Introduction}
\label{sec:intro}

Let $G\to \U(V)$ be a unitary representation of a compact Lie group $G$ on a finite dimensional
Hermitian vector space $(V,\langle\:,\:\rangle)$. By convention, the Hermitian scalar product $\langle\:,\:\rangle$
is assumed to be complex antilinear in the first argument. For the infinitesimal action $d/dt \exp(-t\xi).v_{\mid t=0}$
of an element $\xi$ of the Lie algebra $\mathfrak g$ of $G$ on an element $v\in V$ we write $\xi.v$. The \emph{moment map}
of the representation is the regular map
\begin{eqnarray*}
J:V\to \mathfrak g^*,\quad v\mapsto J_\xi(v):=(J(v), \xi):=\frac{\sqrt{-1}}{2}\langle v,\xi.v\rangle,
\end{eqnarray*}
where $(J(v), \xi)$ stands for the dual pairing of $J(v)\in \mathfrak g^*$ with an arbitrary $\xi\in \mathfrak g$.
Since $J$ is an equivariant map, its zero level $Z:=J^{-1}(0)$ is $G$-saturated, and one can define the
\emph{symplectic quotient} $M_0:=Z/G$.  If $G$ is finite, our convention is to set $Z:=V$ and refer to the
quotient as a \emph{linear symplectic orbifold}.
Note that when $G$ is not discrete,  $0\in \mathfrak g^*$ is a singular value of $J$, and $Z\subset V$ is a closed cone.
In our situation, the symplectic quotient $M_0$ is not a manifold, but a semi-algebraic set stratified by symplectic
manifolds of varying dimensions  \cite{SjamaarLerman}.  We regard it as a \emph{Poisson differential space} with a
\emph{global chart} provided by a Hilbert basis (cf. \cite{FarHerSea}).
When comparing two such spaces, there is a natural notion of equivalence, namely that of a
\emph{$\Z$-graded regular symplectomorphism}. This point of view will be explained in more detail in
Subsection \ref{subsec:regSympBack}, see also \cite{FarHerSea}.

The objective of this paper is to compare symplectic quotients of unitary circle representations with finite
unitary quotients.  Our main result is the following.

\begin{theorem}\label{mainthm}
Let $\Sp^1\to  \U(V)$ be a unitary circle representation such that $V^{\Sp^1}=\{0\}$.
If the dimension of the symplectic quotient $M_0$ is $d>2$, then there cannot exist a $\Z$-graded regular
symplectomorphism of $M_0$ to a quotient of $\C^{n}$ by a finite subgroup $\Gamma <\U_n:=\U(\C^n)$.
\end{theorem}

Note we can assume $n=d/2$.  In \cite{FarHerSea}, one can find a constructive proof of the statement that any $2$-dimensional symplectic
quotient of a unitary torus representation is $\Z$-graded regularly symplectomorphic to some cyclic quotient of $\C$.
We can always assume without loss of generality that $V^{\Sp^1}=\{0\}$, because $V^{\Sp^1}$ is a
symplectic subspace on which the circle acts trivially; see Remark \ref{rem:ZeroWeights}.
So together with \cite{FarHerSea} and the observation that the $0$-dimensional case is trivial, Theorem \ref{mainthm} provides
the complete answer to the question of which linear symplectic circle quotients are $\Z$-graded regularly symplectomorphic
to a finite unitary quotient.

It should be mentioned that some classes of unitary circle representations can be ruled out from our considerations
right from the beginning. If $M_0$ is homeomorphic to a finite symplectic quotient, it is necessarily a rational homology manifold.
Assuming $V^{\Sp^1}=\{0\}$, this is the case precisely when the sign of exactly one weight of the representation differs
from the sign of the others, see \cite{HerbigIyengarPflaum} or \cite[Theorems 3 and 4]{FarHerSea}. In this situation,
the corresponding GIT quotient $V/\!\!/\C^\times$,
which is homeomorphic to $M_0$ by the Kempf--Ness theorem \cite{KempfNess,GWSkempfNess}, is affine space (cf. \cite{WehlauPolynomial}).
It is perfectly legitimate to restrict our attention to these cases. Moreover, in \cite{HerbigSeaton} the authors discovered a Diophantine
condition on the weights (cf. Equation \eqref{eq:Gamma0Diophantine}) that has  to hold if the symplectic circle quotient
is $\Z$-graded regularly symplectomorphic to a finite unitary quotient. Example  calculations (\cite[Section 7]{HerbigSeaton})
indicate that  the `majority' of unitary circle representations violate this condition, even though it holds in infinitely
many cases. The point of this paper is that the remaining cases can be ruled out as well.

Let us outline the plan of the paper. In Section \ref{subsec:regSympBack}, we recall the basics from \cite{FarHerSea} about
how to construct $\Z$-graded regular symplectomorphism from global charts that are provided by the theory of invariant polynomials.
Moreover, we recall in Section \ref{subsec:LaurentBack} some crucial results from \cite{HerbigSeaton} about the Hilbert series
of the $\Z$-graded algebra of regular functions on a symplectic quotient associated to a unitary circle representation.
In Section \ref{sec:generalRed}, we develop tools for understanding the restrictions of global charts for our symplectic
quotients to the closures of orbit type strata.  Specifically, we demonstrate that $\Z$-graded regular symplectomorphisms
between symplectic reduced spaces restrict to $\Z$-graded regular symplectomorphisms between the closures of strata, and in the
relevant cases, the restrictions of global charts are associated to Hilbert bases; see Theorems \ref{thrm:RestrictStratumGeneral}
and \ref{thrm:RestrictStratumFinite}.  On the way, we demonstrate that complex symplectic reduced spaces of
unitary torus representations are normal algebraic varieties; see Theorem \ref{thrm:ComplexRedNormal}.
We rely heavily on the fact that the groups under consideration are abelian or finite, and
we do not know if these results generalize.  If it turns out that global charts are  unique in general,
then Lemma \ref{lem:RestrictStratAlgAbelian} would be a trivial consequence of Theorem \ref{thrm:RestrictStratumGeneral}
for any compact Lie group $G$, and the arguments would simplify considerably.  Based on these tools, we present in
Section \ref{sec:n3Red} an inductive argument that reduces the consideration to the cases when $\dim_{\C} V=3$.
That is, we show that any $\Sp^1$-reduced space that is $\Z$-graded regularly symplectomorphic to a linear symplectic
orbifold must contain the reduced space of an $\Sp^1$-action on $\C^3$ that is as well $\Z$-graded regularly symplectomorphic
to a linear symplectic orbifold; see Corollarey \ref{cor:ReduceToNis3}.  In Section \ref{sec:nis3}, we complete the proof
of Theorem \ref{mainthm} by ruling out case-by-case all finite subgroups of $\U_2$ using the classification of duVal
\cite{DuVal} following the exposition of \cite{CoxeterBook}.

One additional comment is in order.  There is an analogous notion of a $\Z$-graded regular diffeomorphism between Poisson differential
spaces with global charts, defined as a $\Z$-graded regular symplectomorphism where each algebra is equipped
with the trivial Poisson bracket.  A natural question to ask is whether a linear symplectic circle quotient is even $\Z$-graded
regularly diffeomorphic to a linear orbifold.  In this paper, the main techniques used to demonstrate the non-existence of a
$\Z$-graded regular symplectomorphism are restrictions of such a symplectomorphism to the closures of orbit type strata
and comparing the Hilbert series of the $\Z$-graded rings of regular functions.  The Poisson structures on the algebras of
smooth functions in question are only used to establish a bijection between the orbit type strata of $\Z$-graded regularly
symplectomorphic spaces by applying \cite[Proposition 3.3]{SjamaarLerman}; see Section \ref{sec:generalRed}.
If it turns out that the algebra of smooth functions
$\calC^\infty(M_0)$ of a symplectic reduced space $M_0$ determines the orbit type stratification, then the proof of Theorem
\ref{mainthm} would extend to preclude the existence of a $\Z$-graded regular diffeomorphism.


\section*{Acknowledgements}

We would like to thank Srikanth Iyengar for helpful discussions and assistance.
We would like to thank the \emph{Centre for the Quantization of Moduli Spaces}
and Bernhard Lamel for providing the research environment for this project and enabling travel.
The research of HCH was financed by the \emph{Centre for the Quantum Geometry of Moduli Spaces}
(which is funded by the Danish National Research Foundation), the
Austrian Ministry of Science and Research BMWF (Start-Prize Y377) and by the grant GA CR P201/12/G028. CS was supported by
a Rhodes College Faculty Development Grant as well as the E.C. Ellett Professorship in Mathematics.

\section{Background}
\label{sec:background}

The purpose of this section is to recall the definition of the algebra of regular functions on a linear symplectic reduced space
as well as the definition of $\Z$-graded regular symplectomorphism.  We refer to \cite{FarHerSea} for more details.  In addition,
we summarize the formulas from \cite{HerbigSeaton} for the first four coefficients of the Laurent series of a $\Sp^1$-reduced space and
a linear symplectic orbifold for the cases we will need.


\subsection{$\Z$-graded regular symplectomorphisms}
\label{subsec:regSympBack}
As already indicated in the introduction, the symplectic quotient $M_0=Z/G$ of a unitary representation $G\to \U(V)$ is not a
smooth manifold, but a \emph{stratified symplectic space} (this result is due to R. Sjamaar and E. Lerman \cite{SjamaarLerman}).
The symplectic strata $(M_{(H)}\cap Z)/G$ are given by the spaces of $G$-orbits in the intersection of the orbit type strata
$M_{(H)}$ with the zero fiber $Z$ of the moment map. The strata are indexed by  the conjugacy classes  $(H)$ of isotropy
subgroups $H<G$ of the $G$ action on $V$. It has been realized already in \cite{ACG} that $M_0$
carries a natural smooth structure, i.e., the algebra of smooth functions
\[
    \mathcal C^\infty(M_0):=\mathcal C^\infty(V)^G/\mathcal{I}_Z^G,
\]
where $\mathcal{I}_Z^G:=\mathcal C^\infty(V)^G\cap \mathcal{I}_Z$ denotes the invariant part of the ideal
$\mathcal{I}_Z$ of smooth functions vanishing on $Z$.  Equipped with this algebra of smooth functions, $M_0$ has the
structure of a \emph{differential space} (in the sense of Sikorski).  It is moreover pointed out in \cite{ACG}
that $\mathcal C^\infty(M_0)$ inherits a Poisson bracket $\{ , \}$ from $\mathcal C^\infty(V)$.
Hence the triple $(M_0, \calC^\infty(M_0), \{ , \})$ is a \emph{Poisson differential space}; see
\cite[Definition 5]{FarHerSea}.  Using flows of invariant functions, it is shown in \cite{Gonc,SjamaarLerman} that
the symplectic stratification can be reconstructed from the Poisson algebra $(\mathcal C^\infty(M_0),\{\:,\:\})$.
Later, in \cite{LMS}, a notion of equivalence of symplectic quotients has been proposed that generalizes
symplectomorphism between manifolds.  Namely, a map $\Phi: M_0\to N_0$ between two  symplectic quotients
$M_0$ and $N_0$ is called a \emph{symplectomorphism} if $\phi$ is a homeomorphism and its pullback
$\Phi^*:\mathcal C^\infty(N_0)\to \mathcal C^\infty(M_0)$ is an isomorphism of Poisson algebras.
It is clear from what has been said that a symplectomorphism must preserve the symplectic strata.

In the situation of the symplectic quotient of a (unitary) representation $V$, one can talk about symplectomorphisms
that are mediated by polynomials. A language to axiomatizes this idea has been suggested in \cite{FarHerSea}; the
essentials we recall next. The algebra of real-valued regular functions $\R[V]$ on $V$ forms a Poisson subalgebra
of $\mathcal C^\infty(V)$. If we wish to use complex coordinates, we can view $\R[V]$ as the algebra
$\C[V\times \cc V]^-$ of complex polynomials on $V\times \cc V$ that are invariant under complex conjugation $^-$.
By the theorem of Hilbert and Weyl, we can pick a real Hilbert basis $\phi_1,\dots,\phi_k\in\R[V]^G$, i.e. a
complete set of real homogeneous polynomial $G$-invariants.  The Hilbert basis gives rise to an embedding (the
\emph{Hilbert embedding}) of the quotient space $V/G$ into $\R^k$. The image $X$ of the Hilbert embedding is a
closed semialgebraic subset of $\R^k$, and inherits from $\R^k$ a smooth structure $\mathcal C^\infty(X)$ (a
function $f$ on $X$ is defined to be \emph{smooth} if it can be written as a restriction $f=F_{|X}$ of a smooth
function $F$ on $\R^k$). By the Theorem of G. W. Schwarz \cite{GWSdiff} and J. Mather \cite{Mather} the pullback
under $\phi=(\phi_1,\dots,\phi_k)$ provides an isomorphism of Fr\'echet algebras  $\mathcal C^\infty(X)$ to
$\mathcal C^\infty(V)^G$. In other words, $\phi$ provides us with a diffeomorphism of $V/G$ onto $X$.

For each element $\phi_i$, $i=1,\dots,k$ of the Hilbert basis, we introduce a variable $x_i$
and assign to it the degree $\operatorname{deg}(\varphi_i)$. We introduce the free $\Z$-graded
commutative algebra $\R[\bs x]:=\R[x_1,\dots,x_k]$ and consider the algebra homomorphism
\[
    \varphi:\R[\bs x]\to \mathcal C^\infty(M_0)
\]
that sends $x_i$ to the class $\varphi_i$ of $\phi_i$ in $\mathcal C^\infty(M_0)$. This algebra  homomorphism
is not surjective and for most representations not injective. However, it does satisfy the following:
\begin{enumerate}
\item   The image of the algebra homomorphism $\varphi$ is a Poisson subalgebra. We call it the
        \emph{Poisson algebra of regular functions} $\R[M_0]$ on $M_0$. It is isomorphic to the $\Z$-graded
        Poisson algebra  $\R[\bs x]/\Kern(\varphi)$, the bracket being of degree $-2$.
\item   The algebra $\calC^\infty(M_0)$ is $\calC^\infty$-integral over $\R[M_0]$, i.e. every function
        $f\in \calC^\infty(M_0)$ can be written as $f=F\circ (\varphi_1,\dots,\varphi_k)$
        for a suitable choosen $F\in\calC^\infty(\R^k)$.
\item   The image of $\varphi$ in $\mathcal C^\infty(M_0)$ separates points.
\end{enumerate}
Since it is a generalization of a linear coordinate system for a vector space, we call a ring homomorphism
satisfying (1), (2), and (3) a \emph{global chart} for $M_0$.  Two global charts are \emph{equivalent}
if their images in $\mathcal{C}^\infty(M_0)$ coincide.  If the ideal $\Kern(\varphi)$ in $\R[\bs x]$
is homogeneous, then the global chart $\varphi$ is \emph{$\Z$-graded}.  A global chart given
by assigning to each $x_i$ an element of a Hilbert basis for $\R[V]^G$ as above is a \emph{global chart associated
to a Hilbert basis for $\R[V]^G$}.  A global chart provides us with an embedding
of $M_0$ into euclidean space as a closed semialgebraic subset of $X$. It turns out that the choice of the
Hilbert basis is not essential; two global charts associated to different choices of Hilbert bases for
$\R[V]^G$ are easily seen to be equivalent. Whether every global chart is actually equivalent
to a global chart associated to a Hilbert basis is unclear to the authors.

In analogy with coordinate systems on vector spaces, global charts can be used to construct
smooth (Poisson) maps between symplectic quotients. The appropriate tool is provided by
\cite[Theorem 6]{FarHerSea}, the so-called \emph{Lifting Theorem}. It allows to lift certain ($\Z$-graded)
Poisson homomorphisms between Poisson algebras of regular functions
to the Poisson algebras of smooth function, namely those that are compatible with the embeddings into
eucildean space. In analogy with the notion of regular maps between affine varieties, these lifts are be
referred to as \emph{($\Z$-graded)-regular Poisson maps}. An invertible ($\Z$-graded)-regular Poisson
map is an example of a symplectomorphism, called a \emph{($\Z$-graded)-regular symplectomorphism}.

For later reference, let us spell out the dry details on how a $\Z$-graded regular Poisson map is constructed.
Suppose we have two symplectic quotients with $\Z$-graded global charts
$\varphi:\R[\bs x]:=\R[x_1,\dots,x_k]\to \mathcal C^\infty(M_0)$ and
$\psi:\R[\bs y]:=\R[y_1,\dots,y_m]\to \mathcal C^\infty(N_0)$.
Let us assume we have an algebra homomorphism $\lambda: \R[\bs y]\to\R[\bs x]$ compatible with the $\Z$-grading such that
$\lambda(\Kern(\psi))\subset \Kern(\varphi)$, and the induced morphism  $\overline \lambda
:\R[\bs y]/\Kern(\psi)\to\R[\bs x]/\Kern(\varphi)$ is compatible with the Poisson bracket
$\{\:,\:\}$. Note that every morphism of the Poisson algebras $\R[N_0]\to \R[M_0]$ of regular
functions can be represented by such a $\lambda$. Our $\lambda$ is uniquely determined by the
homogeneous polynomials $\lambda_i:= \lambda(y_i)\in\R[\bs x]$ where $i=1,\dots, m$. Let us define
$\vartheta:M_0\to \R^m$ by $\vartheta(\xi):=((\varphi(\lambda_1))(\xi),\dots, (\varphi(\lambda_m))(\xi))$
and say that $\lambda$ is a \emph{arrow} if the image of the map $\underline\psi:N_0\to \R^m$ that sends
$\eta\in N_0$ to $(\psi(y_1))(\eta), \dots, \psi(y_m))(\eta)$ contains $\vartheta(M_0)$. The Lifting Theorem
\cite[Theorem 6]{FarHerSea} says that an arrow $\lambda$ extends uniquely to a morphism of Poisson algebras
$\Lambda:\mathcal C^\infty(N_0)\to \mathcal C^\infty(M_0)$ and that the lift $\Lambda$ is the pullback of the
map that sends $\xi\in M_0$ to $\underline\psi^{-1}(\vartheta(\xi))\in N_0$.

\begin{remark}
\label{rem:ComplexRepresentation}
For later reference, we note that the complexification of the representation of $G$ on
$V \times \cc{V}$ is equal to the $G_\C$-representation $V \times V^\ast$ where $V^\ast$ denotes the dual
representation.  The latter is called the \emph{cotangent lifted $G_\C$-representation}.
Moreover, the complexification $J_\C$ of the moment map $J$ is exactly the moment
map of the cotangent lifted $G_\C$-representation.  The complexification $\mathcal{I}_Z\otimes_\R \C$ of
the vanishing ideal $\mathcal{I}_Z$ of $Z = J^{-1}(0)$ is nothing but the vanishing ideal
$\mathcal{I}_{Z^\C}$ in $\C[V \times V^\ast]$ of the zero locus $Z^\C = J_\C^{-1}(0)$.
Let $\mathcal{I}_{Z^\C}^{G_\C} = \mathcal{I}_{Z^\C} \cap \C[V \times V^\ast]^{G_\C}$,
and then the complexification $\R[M_0] \otimes \C$ of the algebra of regular functions on $M_0$ is given by
$\big(\C[V \times V^\ast]/\mathcal{I}_{Z^\C}\big)^{G_\C} = \C[V \times V^\ast]^{G_\C}/\mathcal{I}_{Z^\C}^{G_\C}$.
\end{remark}


\subsection{Results on symplectic circle quotients}
\label{subsec:LaurentBack}
Here, we specialize the discussion to unitary representations $V$ of the circle $G=\Sp^1$ and recall results from \cite{HerbigSeaton}
regarding the Hilbert series of the $\N$-graded algebra $\R[M_0]$ of regular functions on the symplectic quotient $M_0=Z/G$.
It will be convenient to identify $V$ with $\C^n$ by choosing coordinates $z_1,\dots,z_n$ with respect to which the circle
action is diagonal. Remember that we are
interested in $\C^n$ as a real variety, and so the Poisson algebra of regular functions on $V=\C^n$ is the algebra
$\R[V]\:=\C[z_1,\dots,z_n,\cc z_1,\dots,\cc z_n]^-$ of those complex polynomials in $z_i$ and $\cc z_i$ that
are invariant under complex conjugation ${}^-$. The Poisson bracket is given by $\{z_i,\cc{z}_j\}=-2\sqrt{-1}\delta_{ij}$.
The action is uniquely determined by the moment map
\begin{equation}
\label{eq:MomentMap}
    J:V=\C^n\to \mathfrak g^*=\R, \quad J(\bs z, \cc{\bs z}):=\frac{1}{2}\sum_{i=1}^n a_i\: z_i\cc{z}_i,
\end{equation}
where $a_1,\dots, a_n\in \Z$ are the \emph{weights} of the representation.
We call $A:=(a_1,\dots, a_n)$ the \emph{weight vector}.  The action is effective if and only if the
weights are coprime, i.e., $\gcd(A):=\gcd(a_1,\dots, a_n)=1$. We can restrict to consideration of
the effective case by replacing  $A$ by $A/\gcd(A)$, which does not change the Poisson algebra of regular functions.
Also note that it does no harm to assume that there
are no nontrivial invariants, i.e., all weights are nonzero. If all weights are nonzero but have the
same sign, the reduced a space is a point. Otherwise, by \cite{HerbigIyengarPflaum} we know that the
vanishing ideal $\mathcal{I}_Z$ coincides with the ideal $\mathcal{I}_J$ generated by the moment map.

Recall that the Hilbert series of an $\N$-graded locally finite dimensional vector space
$W=\oplus_{i\ge 0}W_i$  is given by the formal series
\[
    \Hilb_W(x)=\sum_{i= 0}^\infty\dim(W_i)\:x^i\in \Q[\![x]\!].
\]
Let us introduce the \emph{on-shell Hilbert series} of our circle action with weight vector $A$,
\[
    \Hilb_A^\on(x):=\Hilb_{\R[V]/\mathcal{I}_J}(x),
\]
the Hilbert series of the graded algebra $\R[V]/\mathcal{I}_J$. As $\mathcal{I}_Z=\mathcal{I}_J$ in this case, the on-shell
Hilbert series  coincides with the Hilbert series of the $\Z$-graded algebra $\R[M_0]$ of regular functions on
$M_0$. Note that the Hilbert series $\Hilb_A^\on(x)$
does not depend on the signs of the weights. So putting $\alpha_i:=|a_i|$ and writing
$\bs \alpha:=(\alpha_1,\dots, \alpha_n)$ we have $\Hilb_A^\on(x)=\Hilb_{\bs \alpha}^\on(x)$.
We will occasionally use the latter notation for emphasis when we are not concerned with the signs of the weights.

$\Hilb_{A}^{\on}(x)$ is actually a rational function with a pole at $x = 1$ of order $d=2n - 2$,
the Krull dimension of $\R[V]/\mathcal{I}_J$. It therefore admits a Laurent expansion of the form
\begin{equation}
\label{eq:LaurentGeneralReduced}
    \Hilb_{A}^{\on}(x) = \sum\limits_{k=0}^\infty \frac{\gamma_k(A)} {(1 - x)^{d-k}}.
\end{equation}
In \cite[Theorem 5.1]{HerbigSeaton}, the first four Laurent coefficients are computed in terms
of symmetric functions of the $\alpha_j$.  Here, we will need \cite[Corollary 5.2]{HerbigSeaton},
which states that $\gamma_0(A) > 0$ and $\gamma_2(A) = \gamma_3(A) > 0$, as well as the expressions
for the Laurent coefficients when $n = 3$.  It will be convenient to express these coefficients in
terms of the $\alpha_j$ as well as the elementary symmetric polynomials $e_i$ in the $\alpha_j$,
\[
    e_1 =   \alpha_1 + \alpha_2 + \alpha_3,
    \quad
    e_2 =   \alpha_1 \alpha_2 + \alpha_1 \alpha_3 + \alpha_2 \alpha_3,
    \quad\mbox{and}\quad
    e_3 =   \alpha_1 \alpha_2 \alpha_3.
\]
For brevity, we let $\gcd\nolimits_{ij} := \gcd(\alpha_i, \alpha_j)$.

When $n = 3$ and all weights are nonzero, we have
\begin{equation}
\label{eq:Gamma0Reduced}
    \gamma_0(A)
    =
    \frac{\alpha_1 \alpha_2 + \alpha_1 \alpha_3 + \alpha_2 \alpha_3}
    {(\alpha_1 + \alpha_2)(\alpha_1 + \alpha_3)(\alpha_2 + \alpha_3)}
    =
    \frac{e_2}{e_1 e_2 - e_3},
\end{equation}
and
\begin{align}
    \nonumber
    \gamma_2(A)
    &=
    \frac{\alpha_1^2 +  \alpha_2^2 + \alpha_3^2 + \alpha_1 \alpha_2 + \alpha_1 \alpha_3 + \alpha_2 \alpha_3}
        {12(\alpha_1 + \alpha_2)(\alpha_1 + \alpha_3)(\alpha_2 + \alpha_3)}
        \\ \nonumber
        &\quad\quad
            + \frac{\gcd\nolimits_{12}^2-1}{12(\alpha_1+\alpha_2)}
            + \frac{\gcd\nolimits_{13}^2-1}{12(\alpha_1+\alpha_3)}
            + \frac{\gcd\nolimits_{23}^2-1}{12(\alpha_2+\alpha_3)}
    \\ \label{eq:Gamma23Reduced}
    &=
    \frac{1}{12(e_1 e_2 - e_3)}
        \Big[-2e_2 + e_2\big( \gcd\nolimits_{12}^2 + \gcd\nolimits_{13}^2 + \gcd\nolimits_{23}^2 \big)
        \\ \nonumber
        &\quad\quad
            + \gcd\nolimits_{12}^2 \alpha_3^2
            + \gcd\nolimits_{13}^2 \alpha_2^2
            + \gcd\nolimits_{23}^2 \alpha_1^2
        \Big]
\end{align}
Note that $\gamma_1(A) = 0$ in general.

If $\Gamma < \U_{n-1}$ is a finite subgroup, then the Hilbert series $\Hilb_{\R[\C^{n-1}]^\Gamma|\R}(x)$
of real regular functions on the quotient $\C^{n-1}/\Gamma$ has a similar form,
\begin{equation}
\label{eq:LaurentGeneralFinite}
    \Hilb_{\R[\C^{n-1}]^\Gamma|\R}(x)
    =
    \sum\limits_{k=0}^\infty \frac{\gamma_k(\Gamma)} {(1 - x)^{2n - 2-k}}.
\end{equation}
In this case, the first four coefficients of the Laurent series are computed \cite[Theorem 6.1]{HerbigSeaton}
in terms of a collection of primitive pseudoreflections $g_1, \ldots, g_r$ for the action of $\Gamma$ on $\C^{n-1}$,
roughly a minimal generating collection of pseudoreflections.
We recall from \cite[Definition 6.2]{HerbigSeaton} that a set $\{ g_1, \ldots, g_r \}$ of pseudoreflections in $\Gamma$
is \emph{a set of primitive pseudoreflections} if for each pseudoreflection $g \in \Gamma$,
$g = g_i^k$ for $1\leq i \leq r$ and some integer $k$, and if $g_i^k = g_j^\ell \neq e$, then $i = j$ and
$k \equiv \ell \mod |g_i|$.  It is easy to see that such a set always exists though it is not unique.
Note that the diagonal action of $\Gamma$ on $\C^{n-1} \times \overline{\C^{n-1}}$ necessarily has no pseudoreflections.

We have
\begin{equation}
\label{eq:Gamma0123Finite}
    \gamma_0(\Gamma)    =   \frac{1}{|\Gamma|},
    \quad
    \gamma_1(\Gamma)    =   0,
    \quad\mbox{and}\quad
    \gamma_2(\Gamma) = \gamma_3(\Gamma)
        =   \frac{1}{12 |\Gamma|}\sum\limits_{i=1}^r |g_i|^2 - 1.
\end{equation}
In particular, note that $\gamma_2(\Gamma) = \gamma_3(\Gamma) = 0$ if $\Gamma$ contains no pseudoreflections
(in terms of its action on $\C^{n-1}$).

We state two obvious consequences of these computations which will be used in the sequel.  The first was discussed
experimentally in \cite[Section 7]{HerbigSeaton}.

\begin{corollary}[Diophantine Condition on $\gamma_0(A)$]
\label{cor:Gamma0Diophantine}
Let $A = (a_1, \ldots, a_n) \in \Z^n$ be the weight vector associated to a linear $\Sp^1$-action on $\C^n$,
and suppose that $\Hilb_{A}^{\on}(x) = \Hilb_{\R[\C^{n-1}]^\Gamma|\R}(x)$ for some finite $\Gamma < \U_{n-1}$.
Then
\begin{equation}
\label{eq:Gamma0Diophantine}
    \frac{1}{\gamma_0(A)} \in \Z.
\end{equation}
\end{corollary}

Similarly, as $\gamma_0(A) > 0$ for each weight vector $A$ by \cite[Corollary 5.2]{HerbigSeaton},
and as $\gamma_2(\Gamma) = 0$ unless $\Gamma$ contains pseudoreflections, we have the following.

\begin{corollary}
\label{cor:Gamma2Positive}
Let $A = (a_1, \ldots, a_n) \in \Z^n$ be the weight vector associated to a linear $\Sp^1$-action on $\C^n$,
and suppose that $\Hilb_{A}^{\on}(x) = \Hilb_{\R[\C^{n-1}]^\Gamma|\R}(x)$ for some finite $\Gamma < \U_{n-1}$.
Then $\Gamma$ contains at least one pseudoreflection.
\end{corollary}



\section{Restriction of regular symplectomorphisms}
\label{sec:generalRed}

In this section, we consider the restrictions of $\Z$-graded regular symplectomorphisms to closures
of orbit type strata.  We demonstrate in Theorem \ref{thrm:RestrictStratumGeneral} that
such restrictions are themselves $\Z$-graded regular symplectomorphisms.  Moreover, in some cases,
the associated global charts will be shown to be equivalent to global charts associated to Hilbert bases;
see Lemma \ref{lem:RestrictStratAlgAbelian} and Theorem \ref{thrm:RestrictStratumFinite}.

We fix the following notation, which will be used throughout this section.  Let $M_0$ be the symplectic
reduced space associated to a unitary representation $V$ of the compact Lie group $G$ and let $N_0$ be
the symplectic reduced space associated to a unitary representation $W$ of the compact Lie group $K$.
Assume $M_0$ is equipped with the global chart
\[
    \varphi\co \R[\bs{x}] = \R[x_1, \ldots, x_k]\to \calC^\infty(M_0),
    \quad
    x_i \mapsto \varphi_i,
    \quad
    i=1, \ldots, k
\]
associated to a Hilbert basis for $\R[V]^G$, and $N_0$ is equipped with the global chart
\[
    \varphi^\prime\co \R[\bs{y}] = \R[y_1, \ldots, y_\ell]\to \calC^\infty(N_0),
    \quad
    y_i \mapsto \varphi_i^\prime,
    \quad
    i=1, \ldots, \ell
\]
associated to a Hilbert basis for $\R[W]^K$.
Assume further that $\chi\co N_0 \to M_0$ is a $\Z$-graded regular symplectomorphism induced by the arrow
$\lambda\co\R[\bs{x}]\to\R[\bs{y}]$.  We will use $R$ to denote an orbit type stratum of $M_0$
and $S$ to denote an orbit type stratum of $N_0$.

The following is a trivial consequence of the definitions and the fact that, by the above assumptions,
$\overline{\lambda}\co\R[\bs{x}]/\ker\varphi\to\R[\bs{y}]/\ker\varphi^\prime$ is a ring isomorphism
preserving the Poisson bracket.

\begin{lemma}
\label{lem:NewGlobalChart}
Define
\begin{equation}
\label{eq:NewGCpsi}
    \psi\co \R[\bs{x}] \to \calC^\infty(N_0), \quad\quad
    x_i \mapsto \psi_i := \varphi^\prime\circ\lambda(x_i).
\end{equation}
Then $\psi$ is a global chart for $N_0$ that is equivalent to $\varphi^\prime$ and hence associated to
a Hilbert basis for $\R[W]^K$.  With respect to the global charts
$\varphi$ and $\psi$, the $\Z$-graded regular symplectomorphism $\chi$ is induced by the identity arrow
on $\R[\bs{x}]$.
\end{lemma}

This in particular allows us to identify the generators $\varphi_i$ of $\R[M_0]$ with the generators
$\psi_i$ of $\R[N_0]$.

By \cite[Proposition 3.3]{SjamaarLerman}, the stratification of $M_0$ by orbit types, with respect to which it is a symplectic
stratified space, is determined by the Poisson algebra $\calC^\infty(M_0)$.  Because the map $\chi\co N_0 \to M_0$
is a homeomorphism inducing an isomorphism $\widetilde{\lambda}\co\calC^\infty(M_0)\to\calC^\infty(N_0)$ of Poisson algebras,
it follows that $\chi$ induces a bijection between the orbit type strata of $N_0$ and $M_0$.
Specifically, for each orbit type stratum $S$ of $N_0$, $\chi(S)$ is an orbit type stratum of $M_0$.
Then as $\chi(S^{cl}) = \chi(S)^{cl}$ where $^{cl}$ denotes the (topological) closure, it follows that $\chi$ induces a bijection
between closures of orbit type strata.

Now, let $\mathfrak{g}$ denote the Lie algebra of $G$, $J\co V\to\mathfrak{g}^\ast$ the homogeneous quadratic moment map, and
$Z = J^{-1}(0)$ the shell.  Pick an isotropy group $H$ of a point in $Z \subset V$ for the $G$-action on $V$.
Let $V_H$ denote the set of points with isotropy group $H$ and $V^H$ denote the set of points fixed by $H$.
We make some observations following \cite{SjamaarLerman}; note that this case differs from the case treated by Sjamaar--Lerman
in that we apply their arguments to the subspace $V^H = (V_H)^{cl}$, which is a symplectic subspace of $V$
by \cite[Lemma 27.1]{GuilSternSTPhysics}.

The normalizer $N_G(H)$ of $H$ in $G$ acts on $V^H$; this action is not effective unless $H$ is trivial.  That $H$ is an isotropy group
implies that $H$ is the subgroup of $N_G(H)$ acting trivially on $V^H$.  Let $L = N_G(H)/H$ and let $\mathfrak{l}$
denote the Lie algebra of $L$.  The moment map $J \co V \to \mathfrak{g}^\ast$ restricts to a moment map
$J_{V^H} \co V^H \to \mathfrak{l}^\ast$; this is proven for the restriction $J_{V_H}$ of $J$ to $V_H$ in
\cite[page 391--2]{SjamaarLerman}, and the proof for $V^H$ is identical.  Using canonical coordinates
for $V$ that restrict to canonical coordinates for $V^H$, it is easy to see that $J_{V^H}$ is homogeneous
quadratic unless $L$ is finite, in which case $J_{V^H}$ is trivial.
Similarly, that $J_{V^H}^{-1}(0) = V^H \cap Z$ follows from the fact that $J_{V_H}^{-1}(0) = V_H \cap Z$,
and $V_H$ is dense in $V^H$.  The piece $R$ of the stratification of $M_0$ corresponding to the orbit type $(H)$
is given by the (regular) symplectic reduced space $J_{V_H}^{-1}(0)/L$, and similarly
$R^{cl}$ is given by the reduced space $J_{V^H}^{-1}(0)/L$.  Note, however, that $0$ need not be a
regular value for $J_{V^H}$.

Let $\mathcal{I}_{R^{cl}}$ denote the vanishing ideal of $R^{cl}$ in $\calC^\infty(M_0)$.  Then composing with the quotient map
$\calC^\infty(M_0) \to \calC^\infty(M_0)/\mathcal{I}_{R^{cl}} = \calC^\infty(R^{cl})$, one obtains a map
$\varphi_{R^{cl}}\co\R[\bs{x}] \to \calC^\infty(R^{cl})$ that is easily seen to be a global chart for the Poisson
differential space $R^{cl}$.  In particular, as the set $\{ f|_{R} : f \in \calC^\infty(M_0) \}$ is dense in
$\calC^\infty(R)$ by \cite[page 387]{SjamaarLerman}, as the image of $\psi$ is a Poisson subalgebra of $\calC^\infty(M_0)$,
and as $R$ is a piece of the symplectic stratified space $M_0$,
it follows that the image of $\psi_{R^{cl}}$ is a Poisson subalgebra of $\calC^\infty(R^{cl})$.
That $\calC^\infty(R^{cl})$ is $\calC^\infty$-integral over $\R[\bs{x}]$ follows from the integrality of
$\calC^\infty(M_0)$, and that the image of $\varphi_{R^{cl}}$ separates points follows from
the fact that the image of $\varphi$ separates points in $M_0$.  In the same way, letting $S$ be the stratum of $N_0$
such that $\chi(S^{cl}) = R^{cl}$, one obtains a global chart
$\psi^\prime_{S^{cl}}\co\R[\bs{y}] \to \calC^\infty(S^{cl})$, and it is clear that $\lambda$
defines an arrow inducing a $\Z$-graded regular symplectomorphism between $S^{cl}$
and $\chi(S)^{cl}$ that corresponds to the restriction of $\chi$ to $S^{cl}$.
With this, we have proven the following.

\begin{theorem}
\label{thrm:RestrictStratumGeneral}
With the notation as above, for each orbit type stratum $S$ of the symplectic stratified space $N_0$,
$\chi$ restricts to a $\Z$-graded regular symplectomorphism $\chi|_{S^{cl}}$ from the closure $S^{cl}$ to the closure
of an orbit type stratum of the symplectic stratified space $M_0$.  Here, the global charts for
$\chi(S^{cl})$ and $S^{cl}$ are the restrictions of $\varphi$ and $\varphi^\prime$,
respectively, and the arrow inducing $\chi|_{S^{cl}}$ coincides with the arrow inducing $\chi$.
Giving $S^{cl}$ instead the global chart formed by restricting the $\psi$ defined in Equation
\eqref{eq:NewGCpsi}, the arrow inducing $\chi|_{S^{cl}}$ is the identity on $\R[\bs{x}]$.
\end{theorem}

Let $Z_H = J_{V^H}^{-1}(0)$, and identify $R^{cl}$ with the reduced space $Z_H/L$ as above.
It is obvious that restriction to $V^H$ defines a homomorphism $\R[V]^G \to \R[V^H]^L$, and hence
that the image of $\varphi_{R^{cl}}$ is contained in $\R[V^H]^L/\mathcal{I}_{Z_H}^L$.
What is not clear in general, however, is whether $\varphi_{R^{cl}}$ is a surjective map onto
$\R[V^H]^L/\mathcal{I}_{Z_H}^L$.  This is related to the question of whether a global chart on a Poisson differential
space is unique up to equivalence as discussed in Section \ref{sec:intro}.  Certainly, the algebra $\R[V^H]^L$ may admit
a proper subalgebra that separates points, see e.g. \cite[Section 2.3.2]{DerskenKemperBook}, but the authors are
currently unaware of whether such an algebra could satisfy the Poisson and integrality conditions.
Fortunately, however, the surjectivity of $\varphi_{R^{cl}}$ onto $\R[V^H]^L/\mathcal{I}_{Z_H}^L$ is clear
in the cases relevant to this investigation.  The rest of this section is devoted to verifying this fact.

\begin{lemma}
\label{lem:RestrictStratAlgAbelian}
Suppose $G$ is abelian and $\varphi$ is a global chart associated to a Hilbert basis for $\R[V]^G$.
Then the image of $\varphi_{R^{cl}}$ is equal to $\R[V^H]^L/\mathcal{I}_{Z_H}^L$.
\end{lemma}
\begin{proof}
Choose an isotropy group $H$ for the $G$-action on $V$ and use the notation as above.
As $G$ is abelian, $V^H$ is $G$-invariant, and $L = G/H$.  Then the restriction map
$\C[V \times  V^\ast]^{G_\C} \to \C[(V \times  V^\ast)^H]^{G_\C}$ is surjective
by \cite[Corollary 2.2.9]{DerskenKemperBook}.
Taking the fixed elements by complex conjugation, the restriction map $\R[V]^G \to \R[V^H]^G = \R[V^H]^L$ is
surjective. Therefore, restricting a Hilbert basis for $\R[V]^G$ yields a complete (not necessarily minimal)
set of invariants for the $L$-action on $V^H$, completing the proof.
\end{proof}

We assume for the remainder of this section that $G$ is a torus.

Recall \cite{GWSlifting} that if $Y$ is a $G_\C$-variety, then $Y$ has \emph{finite principal isotropy groups (FPIG)}
if there is a point $y \in Y$ such that the orbit $G_\C y$ is closed and the isotropy group of $y$ in $G_\C$ is
$0$-dimensional.  In our case, as $G$ is abelian, it is easy to see that the isotropy group of a point
in the principal orbit type of the $G$-space $V$ is the kernel of the $G$-action.  Indeed, choose a point $v$
in the principal orbit type stratum of $V$ as a $G$-space, and let $H$ be the isotropy group of $v$.  Then as
$H$ is the isotropy group of every point in the principal orbit type stratum, and as the
principal orbit type stratum is dense in $V$, it follows that $H$ is the kernel of the $G$-action.
If we assume further that the $G$-action is effective, it follows that there is a dense subset of $V$
consisting of points with trivial isotropy.

It may happen, though, that each such point in $V$ has infinite $G_\C$-isotropy or fails to have closed
$G_\C$-orbits so that $V$ may not have FPIG as a $G_\C$-variety.  However, this only occurs
in somewhat artificial situations.

\begin{example}
\label{ex:ReductionFPIG}
Let $\C^\times$ act on $\C$ with weight vector $(1)$.  It is easy to see that each point
$z \neq 0$ has trivial isotropy group.  However, the only point with a closed orbit
is $0$ with isotropy $\C^\times$ so that $\C$ does not have FPIG as a $\C^\times$-variety.

Note, though, that the moment map is given by $J(z) = z_1 \bar{z_1}/2$,
see Equation \eqref{eq:MomentMap}, so that $J^{-1}(0) = \{ 0 \}$.  Hence, as $J^{-1}(0)$ is contained
in a proper subspace of $\C$ (the trivial subspace), we may restrict our attention to this subspace without changing
the reduced space as a Poisson differential space.  That is, we may begin with the trivial action of
$\C^\times$ on a point to yield the same reduced space.  Effectivising the action, we are left with a point
with trivial group action, which clearly has FPIG.
\end{example}

The reason for the failure of FPIG in Example \ref{ex:ReductionFPIG} and the method of correcting it are
in some sense the general situation for unitary torus representations.  Specifically, we have the following.

\begin{lemma}
\label{lem:ReducedCoordinateRestrict}
With the notation as above, assume $G$ is a torus.  Choose coordinates $(z_1, \ldots, z_n)$ for $V$
with respect to which the $G$-action is diagonal, let $I$ denote the subset of $\{ 1,\ldots,n\}$
such that $z_i = 0$ for every point $(z_1,\ldots,z_n) \in J^{-1}(0)$, and let $V^\prime$
denote the subspace of $V$ defined by $z_i = 0 \;\forall i \in I$.  Then the symplectic
reduced spaces of the $G$-representations $V$ and $V^\prime$ coincide.  Moreover,
$J^{-1}(0)$ has nonempty intersection with the principal orbit type stratum of $V^\prime$.
\end{lemma}
\begin{proof}
The facts that $V^\prime$ is $G$-invariant and the two reduced spaces coincide are demonstrated in the
proof of \cite[Theorem 4]{FarHerSea}.  Moreover, from the description of the weight matrices for the
$G$-actions on $V$ and $V^\prime$ in the same proof, it is easy to see that every point in $V^\prime$
whose coordinates $z_j$ are all nonzero and have principal isotropy type.  Noting that $J^{-1}(0)$
contains points in $V^\prime$ with all nonzero coordinates by construction completes the proof.
\end{proof}

\begin{remark}
\label{rem:WehlauStability}
Using the Kempf--Ness theorem \cite{KempfNess,GWSkempfNess}, it is easy to see that the definition
of $V^\prime$ given in Lemma \ref{lem:ReducedCoordinateRestrict} coincides with that of \cite[Lemma 2]{WehlauPopov}.
By this result, if the $G$- and hence $G_\C$-actions are effective, it follows that the $G_\C$-action on $V^\prime$
is \emph{stable}, i.e. the union of the closed $G_\C$-orbits is dense in $V^\prime$.
\end{remark}

As a consequence of Lemma \ref{lem:ReducedCoordinateRestrict}, we may assume without loss of generality
that $J^{-1}(0)$ contains with principal isotropy group.  We then have the following.

\begin{theorem}
\label{thrm:ComplexRedNormal}
Let $G$ be a torus.  Then $\R[M_0] \otimes \C$ is normal.
\end{theorem}
\begin{proof}
We assume without loss of generality that the $G$-action on $V$ is effective and, by Lemma \ref{lem:ReducedCoordinateRestrict},
that there is a point $v \in J^{-1}(0)$ with principal $G$-isotropy group; the effective assumption then implies that
the $G$-isotropy group of $v$ is trivial.  Then as $J^{-1}(0)$ is a Kempf-Ness set, it follows from
\cite[Theorem (4.2)]{GWSkempfNess} that $v$ has trivial $G_\C$-isotropy and closed $G_\C$-orbit.
Hence $V$ has FPIG as a $G_\C$-variety.

By \cite[Theorem (0.4)(4)]{GWSlifting}, $V$ is \emph{$2$-large} (see \cite{GWSlifting} for the definition),
and then by \cite[Theorem 2.2 (2)]{HerbigSchwarz},
$\C[V \times V^\ast]/\mathcal{I}_{Z^\C}$ is normal.  Finally, by \cite[Proposition 6.4.1]{BrunsHerzog}, the quotient
$\big(\C[V \times V^\ast]/\mathcal{I}_{Z^\C}\big)^{G_\C} = \C[V \times V^\ast]^{G_\C}/\mathcal{I}_{Z^\C}^{G_\C}$
is normal, completing the proof; confer Remark \ref{rem:ComplexRepresentation}.
\end{proof}

\begin{remark}
\label{rem:2largeDef}
The reader is cautioned that the definition of $2$-large in \cite{GWSlifting} is stronger than the definition
used in \cite{HerbigSchwarz}.  Specifically, the definition in \cite{GWSlifting} requires FPIG, while that in
\cite{HerbigSchwarz} does not.  This causes no issue in the above argument as FPIG has been established.
\end{remark}

\begin{remark}
\label{rem:NormalityGeneral}
Theorem \ref{thrm:ComplexRedNormal} essentially demonstrates that the complex symplectic reduced space of an
effective linear torus action is normal by proving that the zero fiber of the complex moment map is a normal
variety.  For a reductive group $G_\C$ and a $G_\C$-variety $X$, the normality of $X$ is sufficient though
not necessary for the normality of $X/\!\!/G_\C$, see \cite[Section 7]{CrawleyBoevey}.
By Serre's Criterion \cite[Theorem 23.8]{MatsumuraBook}, a Cohen--Macaulay ring is normal if and only if it is
regular in codimension $1$.  In our case, $\C[V\times V^\ast]^{G_\C}/\mathcal{I}_{Z^\C}^{G_\C}$
is always Cohen--Macaulay.  To see this, note that $\C[V\times V^\ast]^{G_\C}$ is Cohen--Macaulay by
\cite[Corollary 1]{HochsterTori}, and then the quotient ring is Cohen--Macaulay by
\cite[Theorem 2.1.3]{BrunsHerzog}.
Hence, $\C[V\times V^\ast]^{G_\C}/\mathcal{I}_{Z^\C}^{G_\C}$ is normal if and only if the singular set
of the complex symplectic reduced space has codimension $\geq 2$.  We do not know of an independent proof
of this codimension condition.
\end{remark}

Now assume that $K$ is finite so that $N_0 = W/K$ is a linear symplectic orbifold.  We have the following,
completing this section.

\begin{theorem}
\label{thrm:RestrictStratumFinite}
With the notation as above, assume that $G$ is a torus and $K$ is finite.  For each isotropy group $\Gamma$
of the $K$-action on $W$ and corresponding orbit type stratum $S$ of the linear symplectic orbifold $N_0$,
the global chart for $S^{cl}$ given by the restriction of  $\psi$ (or $\varphi^\prime$) is equivalent
to a global chart associated to Hilbert bases for $\R[W^\Gamma]^{N_K(\Gamma)/\Gamma}$.
\end{theorem}
\begin{proof}
Let $\Gamma$ be an isotropy group for the action of $K$ on $W$.  The corresponding stratum $S$ is given by $W_{\Gamma}/N_K(\Gamma)$,
and its closure $S^{cl}$ is $W^\Gamma/N_K(\Gamma)$.  As above, the image of $\psi_{S^{cl}}$ is
a subalgebra $\calA$ of $\R[W^\Gamma]^{N_K(\Gamma)/\Gamma}$.
Let $R = \chi(S)$, and then $R$ is an orbit type stratum of $M_0$ by Theorem \ref{thrm:RestrictStratumGeneral}.
Let $H$ be the isotropy group of a point in $V$ whose orbit is in $R$; note that $H$ does not depend on the choice
of point as $G$ is abelian.
By Theorem \ref{thrm:RestrictStratumGeneral}, the identity arrow on $\R[\bs{x}]$ induces an isomorphism between
the image of $\varphi_{R^{cl}}$ and $\calA$, and by Lemma \ref{lem:RestrictStratAlgAbelian}, the image of
$\varphi_{R^{cl}}$ is equal to $\R[V^H]^L/\mathcal{I}_{Z_H}^L$.  It follows that the identity arrow
on $\R[\bs{x}]$ induces an isomorphism $\kappa\co\R[V^H]^L/\mathcal{I}_{Z_H}^L \to \calA$.
Tensoring with $\C$, we have an isomorphism
$\kappa_\C\co (\R[V^H]^L/\mathcal{I}_{Z_H}^L)\otimes_\R \C\to\calA \otimes_\R \C=:\calA_{\C}$.

Noting that $(V\times\bar{V})^H \otimes_\R\C = (V \times V^\ast)^{H_\C}$, we have by \cite[Proposition 5.8(1)]{GWSliftingHomotopies}
that $\R[V^H]^L = \big(\C[(V \times \bar{V})^H]^{L_\C}\big)^-$, and
$\R[V^H]^L\otimes_\R\C = \C[(V \times V^\ast)^{H_\C}]^{L_\C}$.
Then we have $(\R[V^H]^L/\mathcal{I}_{Z_H}^L)\otimes_\R \C = \C[(V \times V^\ast)^{H_\C}]^{L_\C}/\mathcal{I}_{Z_H^\C}^{L_\C}$.
In the same way, we may consider $\calA_{\C}$ as a subalgebra of
$\C[(W \times W^\ast)^\Gamma]^{N_K(\Gamma)} = (\R[W^\Gamma]^{N_K(\Gamma)/\Gamma})\otimes_\R\C$.
That is, $\kappa_\C$ can be viewed as an isomorphism
\[
    \kappa_\C\co
    \C[(V \times V^\ast)^H]^{L_\C}/\mathcal{I}_{Z_H^\C}^{L_\C}
    \to
    \calA_{\C},
\]
and as $\C[(V \times V^\ast)^H]^{L_\C}/\mathcal{I}_{Z_H^\C}^{L_\C}$ is normal by
Lemma \ref{thrm:ComplexRedNormal}, $\calA_{\C}$ is normal.

Finally, we claim that the algebra $\calA_\C$ is separating in $\C[(W \times W^\ast)^\Gamma]^{N_K(\Gamma)}$
in the sense of \cite[Definition 2.3.8]{DerskenKemperBook}, see also \cite{DufresneThesis}.  That is, for
$x, y \in (W\times W^\ast)^\Gamma$, if there is an $f \in \C[(W \times W^\ast)^\Gamma]^{N_K(\Gamma)}$ such that
$f(x) \neq f(y)$, then there is a $g \in \calA_\C$ such that $g(x) \neq g(y)$.
Because $\C[(W \times W^\ast)^\Gamma]^{N_K(\Gamma)}$ separates points
in the complex geometric quotient $(W \times  W^\ast)^\Gamma/(N_K(\Gamma))$ by
\cite[Theorem 1.1, Amplification 1.3]{MumfordFogarty}, it follows that this is equivalent to
$\calA_\C$ separating $N_K(\Gamma)$-orbits in $(W \times  W^\ast)^\Gamma$.

To see this, first note that the map $\eta\co (W \times  W^\ast)^\Gamma/(N_K(\Gamma)) \to (W\times W^\ast)/K$
induced by the embedding $(W \times  W^\ast)^\Gamma \to W\times W^\ast$ is injective, see \cite[page 100]{KLMR}.
Specifically, because $K_{hx} = hK_x h^{-1}$ for $x \in W \times W^\ast$ and $h \in K$, $\eta$
is obviously an injective regular map on the open dense subset $(W \times  W^\ast)_\Gamma/(N_K(\Gamma))$ of
$(W \times  W^\ast)^\Gamma/(N_K(\Gamma))$ and hence birational.  Then the normality of
$(W \times  W^\ast)^\Gamma/(N_K(\Gamma))$ implies that $\eta$ is an isomorphism of algebraic varieties;
see \cite[II. Section 2.7]{EncMathSciAlgeGeomI}.
Choosing $x, y \in (W \times  W^\ast)^\Gamma$ with distinct orbits $N_K(\Gamma)x \neq N_K(\Gamma)y$, it
follows that $Kx \neq Ky$ in $W \times  W^\ast$.  Then as $\C[W\times W^\ast]^K$ separates
$K$-orbits in $W\times W^\ast$ (again by \cite{MumfordFogarty}),
there is an $F \in \C[W\times W^\ast]^K$ such that $F(x) \neq F(y)$.

Again by \cite[Proposition 5.8(1)]{GWSliftingHomotopies},
the isomorphism $\overline{\lambda}\co\R[V]^G/\mathcal{I}_Z \to \R[W]^\Gamma$
given by the arrow $\lambda$ induces an isomorphism
$\overline{\lambda}_\C\co \C[V\times V^\ast]^{G_\C}/\mathcal{I}_{Z^\C} \to \C[W\times W^\ast]^\Gamma$.
It follows that there is a $g \in \C[V\times V^\ast]^{G_\C}$ such that
$\overline{\lambda}(g \mathcal{I}_{Z^\C}) = F$.  Let $g_{H}$ denote the restriction of $g$ to
$(V\times V^\ast)^H$, and then $\kappa_\C(g_{H}\mathcal{I}_{Z_H^\C}^{L_\C})$ is equal to the restriction
$f:=F_{|(W\times W^\ast)^\Gamma}$ of $F$ to $(W\times W^\ast)^\Gamma$.  Then $f(x) = F(x) \neq F(y) = f(y)$
so that $\calA_\C$ is separating.

By \cite[Theorem 2.3.12]{DerskenKemperBook}, $\C[(W \times  W^\ast)^\Gamma]^{N_K(\Gamma)}$ is equal to
the normalization of $\calA_\C$.  As $\calA_\C$ is normal, we obtain
$\C[(W \times  W^\ast)^\Gamma]^{N_K(\Gamma)} = \calA_\C$.  This clearly implies
$\R[W^\Gamma]^{N_K(\Gamma)/\Gamma} = \calA$ so that $\psi_{S^{cl}}$ is surjective, completing the proof.
\end{proof}

\section{Reduction to the case $n=3$}
\label{sec:n3Red}

In this section, we will use the results of Section \ref{sec:generalRed} to demonstrate that any reduced
space $M_0$ satisfying the hypotheses of Theorem \ref{mainthm} and $\Z$-graded regularly symplectomorphic
to a linear symplectic orbifold $\C^{n-1}/\Gamma$ for a finite subgroup $\Gamma < \U_{n-1}$
contains as a subset the reduced space associated to an $\Sp^1$-action on $\C^3$,
which is then $\Z$-graded regularly symplectomorphic to $\C^2/\Gamma^\prime$ for some finite $\Gamma^\prime < \U_2$.
In Section \ref{sec:nis3}, we will show that no such $\Gamma^\prime$ can exist.

Let $A = (a_1, \ldots, a_n)$ be a weight vector for a unitary action of $\Sp^1$ on $\C^n$.
We make the following assumptions.
\begin{itemize}
\item[(i)]      There is a $\Z$-graded regular symplectomorphism $\chi\co M_0 \to \C^{n-1}/\Gamma$
                from the reduced space $M_0$ associated to $A$ and a linear symplectic orbifold $\C^{n-1}/\Gamma$
                for a finite subgroup $\Gamma < \U_{n-1}$.
\item[(ii)]     Each weight $a_i \neq 0$, i.e. $(\C^n)^{\Sp^1} = \{ 0 \}$.
\item[(iii)]    The weight $a_1$ is the only negative weight.
\end{itemize}

\begin{remark}
\label{rem:ZeroWeights}
Recall that permuting the weights of the weight vector $A$ does not change the $\Z$-graded regular symplectomorphism
class of the associated reduced space $M_0$,
and suppose $A = (a_1, \ldots, a_k, 0, \ldots, 0)$ with $a_i \neq 0$ for $i \leq k$.  It is easy to see that
$M_0$ is $\Z$-graded regularly symplectomorphic to $N_0 \times \C^{n-k}$ where $N_0$ is the reduced space
associated to the weight vector $(a_1, \ldots, a_k)$.  As
$\R[\C^n]^{\Sp^1}$ contains $2(n-k)$ linear invariants and the moment
map $J$ is homogeneous quadratic, $\R[M_0]$ must also contain
$2(n-k)$ linear invariants.  Then it must be that $\R[\C^{n-1}]^\Gamma$ contains $2(n-k)$ linear invariants as well,
implying that $(\C^{n-1})^\Gamma$ is of complex dimension $n-k$.  Decomposing $\C^{n-1}$ into a product
$\C^{k-1} \times (\C^{n-1})^\Gamma \cong \C^{k-1} \times \C^{n-k}$, it is easy to see that $N_0$ is $\Z$-graded
regularly symplectomorphic to $\C^{k-1}/\Gamma^{\prime\prime} \times \C^{n-k}$ for some $\Gamma^{\prime\prime} \in \U_{k-1}$.
It follows that $N_0$ is $\Z$-graded regularly  symplectomorphic to $\C^{k-1}/\Gamma^{\prime\prime}$.
Hence assumption (ii) introduces no loss of generality.
\end{remark}

\begin{remark}
\label{rem:NegativeA1}
By \cite[Theorems 3 and 4]{FarHerSea} (see also \cite[Proposition 3.1]{HerbigIyengarPflaum}), we have that $M_0$ is
a rational homology manifold if and only if $A$ does not contain two or more positive and two or more negative weights.
Note that as each element $g \in \Gamma$ preserves the complex structure of $\C^{n-1}$ and hence the induced orientation of the
underlying real vector space $\R^{2(n-1)}$, the complex orbifold $\C^{n-1}/\Gamma$ is necessarily a rational homology manifold
by \cite[4.2.4]{Haefliger}. If all weights have the same sign, then $M_0$ is a point, and multiplying $A$ by $-1$ or permuting
the weights does not change the reduced space $M_0$.  Hence we may assume (iii) with no loss of generality as well.
\end{remark}

Let $t \in \Sp^1$, and then it is easy to see that $t$ fixes a nonzero point $\bs{z} = (z_1, \ldots, z_n) \in \C^n$ if and only if
$t$ is a $k$th root of unity for some integer $k$, and $k$ divides $a_i$ for each $i$ such that $z_i \neq 0$.  With this in mind,
for a subset $I \subseteq \{ 1, 2, \ldots, n \}$,
let $V_I = \{ (z_1, \ldots, z_n) \in \C^n : z_i \neq 0 \Leftrightarrow i \in I \}$,
and let $V_I^{cl} = \{ (z_1, \ldots, z_n) \in \C^n : z_i \neq 0 \Rightarrow i \in I \}$
denote the closure of $V_I$ in $\C^n$.  We let $V_\emptyset = V_\emptyset^{cl} = \{ 0 \}$.
Then, given a $t \in \Sp^1$ of finite order $k$, letting $I_k = \{ i : k|a_i \}$, we have that
$(\C^n)^t = V_{I_k}^{cl}$.  Moreover, for the cyclic subgroup $\Z_k$ of $\Sp^1$, it is easy to see
that $(\C^n)_{\Z_k}$, the set of points with isotropy group equal to $\Z_k$, is given by
$\{ \bs{z} : \gcd\{a_i : z_i \neq 0\} = k \}$, and the closure $(\C^n)_{\Z_k}^{cl}$ of this orbit type is $V_{I_k}^{cl}$.

By Equation \eqref{eq:MomentMap} and assumption (iii), it is clear that $V_I \cap Z \neq \emptyset$ if and only if
$I = \emptyset$ or $\{ 1 \} \subsetneq I$.  Hence, an orbit type $(\C^n)_{\Z_k}$ intersects $Z$ if and only if
there is an $I \subseteq \{ 1, 2, \ldots, n \}$ such that $\{ 1 \} \subsetneq I$ and $\gcd\{ a_i : i \in I \} = k$.

As a consequence of the assumption (i), we have that the Hilbert series $\Hilb_{A}^{\on}(x)$ and $\Hilb_{\R[\C^{n-1}]^\Gamma|\R}(x)$
coincide.  Then by Corollary \ref{cor:Gamma2Positive}, $\Gamma$ must contain at least one pseudoreflection.
It follows that $\C^{n-1}/\Gamma$ contains a (complex)
codimension-$1$ orbit type stratum, and that a corresponding isotropy group is generated by a pseudoreflection $g \in \Gamma$.
Choosing a basis for $\C^{n-1}$ such that $g = \diag(1, \ldots, 1, \lambda)$ for a root of unity $\lambda$, and identifying
$\C^{n-2}$ with the span of the first $n-2$ elements of this basis, this orbit type stratum is given by
$\C^{n-2}/(N_\Gamma(\langle g\rangle)/\langle g\rangle)$.  Note that the elements of
$N_\Gamma(\langle g\rangle)$ clearly fix the Hermitian product on $\C^{n-2}$ (restricted from $\C^{n-1}$) so that
$(N_\Gamma(\langle g\rangle)/\langle g\rangle) < \U_{n-2}$.

By Theorem \ref{thrm:RestrictStratumGeneral}, $M_0$ must also contain a (complex) codimension-$1$ orbit type stratum.
Considering the orbit types of the $\Sp^1$-action on $\C^n$, the closure of this stratum must be given by $V_{\{ i \}^c}^{cl}$
for some $i \neq 1$, where $^c$ denotes the complement (or simply the origin $\{ 0 \}$ if $n = 2$).
If $\varphi\co\R[\bs{x}] \to\calC^\infty(M_0)$ is a global chart for $M_0$ associated
to a Hilbert basis for $\R[\C^n]^{\Sp^1}$, then by Lemma \ref{lem:RestrictStratAlgAbelian}, $\varphi$ restricts to a
global chart for $(V_{\{ 1 \}^c}^{cl} \cap Z)/\Sp^1$ equivalent to a global chart associated to a Hilbert basis for
$\R[V_{\{ 1 \}^c}^{cl}]^{\Sp^1}$ with the restricted action.  Note that in this case, it is easy to see directly that
$(V_{\{ 1 \}^c}^{cl} \cap Z)/\Sp^1$ is the reduced space associated to the $\Sp^1$-action on $\C^{n-1}$ whose weight vector
is formed by removing $a_i$ from $A$.  Finally, Theorem \ref{thrm:RestrictStratumFinite} yields a $\Z$-graded regular
symplectomorphism between $(V_{\{ 1 \}^c}^{cl} \cap Z)/\Sp^1$ and $\C^{n-2}/(N_\Gamma(\langle g\rangle)/\langle g\rangle)$
where the global charts are associated to Hilbert bases for the respective actions.  With this, we have established the
following.

\begin{proposition}
\label{prop:ReductionStep}
Let $n \geq 2$, let $A$ be a weight vector satisfying (i), (ii), and (iii), and let $M_0$ denote the associated reduced space.
Then $M_0$ has a complex codimension-$1$ orbit type $S$.
Moreover, the closure $S^{cl}$ is itself the reduced space associated to a linear representation of
$\Sp^1$ on $\C^{n-1}$ satisfying (ii) and (iii), and $\chi$ restricts to a $\Z$-graded regular symplectomorphism
to the quotient of $\C^{n-2}$ by a finite subgroup of $\U_{n-2}$.
\end{proposition}

By repeated application of Proposition \ref{prop:ReductionStep}, the following is immediate.

\begin{corollary}
\label{cor:ReduceToNis3}
Let $n \geq 3$, let $A$ be a weight vector satisfying (i), (ii), and (iii), and let $M_0$ denote the associated reduced space.
Then $M_0$ contains as the closure of an orbit type a linear symplectic quotient of $\C^3$ by $\Sp^1$ satisfying (ii) and (iii)
that is $\Z$-graded regularly symplectomorphic to a linear symplectic orbifold $\C^2/\Gamma^\prime$ where $\Gamma^\prime < \U_2$
is finite and $\C^2/\Gamma^\prime$ is equipped with a global chart associated to a Hilbert basis for $\R[\C^2]^{\Gamma^\prime}$.
\end{corollary}

Of course, Proposition \ref{prop:ReductionStep} also guarantees an orbifold stratum of complex dimension $1$
given by the reduced space of an $\Sp^1$-action on $\C^2$.  Such a reduced space is always an orbifold by \cite[Theorem 7]{FarHerSea},
explaining why we stop at $n = 3$.
Note, however, that many examples of linear $\Sp^1$ reduced spaces are already excluded as non-orbifolds
by Proposition \ref{prop:ReductionStep}.  In particular, for a weight vector $A$ satisfying (i), (ii), and (iii),
if $M_0$ has a complex codimension-$1$ orbit type, it must be that there is an $i_1\neq 1$ such
that $\gcd \{ a_j : j \neq i_1\}> \gcd(a_1, \ldots, a_n)$.
For the closure of this orbit type in turn to contain a
complex codimension $1$-orbit type, it must be that there is an $i_2 \neq 1, i_1$ such that
$\gcd\{a_j : j \neq i_1, i_2 \} > \gcd\{ a_j : j \neq i_1\}$.
Then by induction, we obtain the following.

\begin{corollary}
\label{cor:ReductionExcludes}
Let $n \geq 3$ and let $A$ be a weight vector satisfying (i), (ii), and (iii).
Then for some ordering
of the weights $a_2, \ldots, a_n$, we have that for $2 \leq i \leq n-1$, $\gcd(a_1, a_2, \ldots, a_i)$
does not divide $a_\ell$ for any $\ell > i$.
\end{corollary}

\begin{example}
\label{ex:ReductionExcludes}
Consider the weight vector $A = (-3, 6, 12, 4)$.  Using the Diophantine condition in Corollary \ref{cor:Gamma0Diophantine}
and the formula for $\gamma_0(A)$ in
\cite[Section 7]{HerbigSeaton}, we have that $\gamma_0(A) = 1/21$ is the reciprocal of an integer,
so that the associated reduced space $M_0$ may be $\Z$-graded regularly symplectomorphic to an orbifold.
However, by inspection, the only complex codimension-$1$ orbit type stratum in $M_0$ corresponds to $z_4 = 0$
with isotropy group $\Z_3$.  The closure of this orbit type is the reduced space associated to the weight vector $(-3, 6, 12)$.
This representation is not effective, and the corresponding effective representation
has weight vector $(-1, 2, 4)$.  However, as $\gcd(1,2) = \gcd(1,4) = \gcd(1,2,4) = 1$, the corresponding reduced
space does not have a complex codimension-$1$ orbit type stratum.  It follows that $M_0$ (equipped with a global chart
associated to a Hilbert basis for $\R[\C^4]^{\Sp^1}$) is not $\Z$-graded regularly symplectomorphic to an orbifold.
\end{example}

\section{The case $n=3$}
\label{sec:nis3}

Let $V = \C^3$ be equipped with the $\Sp^1$-action with weight vector $A = (a_1, a_2, a_3)$
such that each $a_i$ is nonzero and $\gcd(a_1,a_2,a_3) = 1$.
Assume further that the symplectic reduced space $M_0 = J^{-1}(0)/\Sp^1$ with global chart associated to a Hilbert
basis for $\R[\C^3]^{\Sp^1}$ is $\Z$-graded regularly
symplectomorphic to a linear symplectic orbifold $\C^2/\Gamma$ for $\Gamma < \U_2$ finite.
In this section, we will show that no such $A$ and $\Gamma$ exist, which along with
Corollary \ref{cor:ReduceToNis3} derives the contradiction that completes the proof of Theorem \ref{mainthm}.

In Subsection \ref{subsec:HilbM0}, we collect some restrictions on the Hilbert series of the algebra of
regular functions $\R[M_0]$.  In Subsection \ref{subsec:HilbC2Gamma}, we make similar observations about
the Hilbert series of $\R[\C^2]^\Gamma$.  In Subsection \ref{subsec:TypesU2},
we use these results and the classification of finite subgroups of $\U_2$ in \cite{DuVal,CoxeterBook}
to exclude each possible $\Gamma$.

\begin{remark}
\label{rem:PoissonStrucTypeI}
Recall that in this paper, the Poisson structures on the algebras of regular functions are only used to establish
a bijection between orbit type strata of $\Z$-graded regular symplectomorphic reduced spaces; see Section
\ref{sec:generalRed}.  We point out that in this section, this bijection is used again in Lemmas
\ref{lem:TypeI2pseudoreflections} an \ref{lem:TypeI1pseudoreflection} and hence is required to rule out
groups of Type I.
\end{remark}


\subsection{The Hilbert series of $M_0$}
\label{subsec:HilbM0}

Let $A = (a_1, a_2, a_3)$ with each $a_i \neq 0$ as above and recall that $\alpha_j = |a_j|$ and
$\bs{\alpha} = (\alpha_1, \alpha_2, \alpha_3)$.  The weight vector $A$ is
\emph{generic} if $\alpha_i \neq \alpha_j$ for $i \neq j$ and \emph{degenerate} otherwise, see
\cite[Definition 2.2]{HerbigSeaton}.  Our first task is to exclude degenerate weight vectors,
which is a consequence of the Diophantine condition of Corollary \ref{cor:Gamma0Diophantine}
on $\gamma_0(\bs{\alpha})$ as defined by Equation \eqref{eq:LaurentGeneralReduced}.

\begin{lemma}
\label{lem:Reduced3Degenerate}
Let $A = (a_1, a_2, a_3)$ be the weight vector associated to a linear $\Sp^1$-action on $\C^3$
such that each $a_i$ is nonzero and $\gcd(a_1, a_2, a_3) = 1$.  If $A$ is degenerate,
then the associated reduced space $M_0$ is not $\Z$-graded regularly symplectomorphic to a linear
symplectic orbifold $\C^2/\Gamma$ for finite $\Gamma < \U_2$.
\end{lemma}
\begin{proof}
Suppose $A$ is degenerate.  Then either $\alpha_1 = \alpha_2 = \alpha_3$, in which case $\bs{\alpha} = (1,1,1)$,
or up to permuting weights, $\bs{\alpha} = (\alpha, \alpha, \beta)$ for positive coprime integers $\alpha$ and $\beta$.

The case $\bs{\alpha} = (1,1,1)$ is referred to as \emph{completely degenerate} in \cite[Section 5.3]{HerbigSeaton},
where the corresponding $\gamma_0$ is computed to be $3/8$.  As this is not the reciprocal of an integer, we
have by Corollary \ref{cor:Gamma0Diophantine} that the corresponding $M_0$ is not $\Z$-graded regularly
symplectomorphic to a linear symplectic orbifold.

So assume now that $\bs{\alpha} = (\alpha, \alpha, \beta)$ for positive coprime integers $\alpha$ and $\beta$,
and then by Equation \eqref{eq:Gamma0Reduced}, we have
\[
    \frac{1}{\gamma_0(\bs{\alpha})}
        =   \frac{2\alpha(\alpha + \beta)^2}{\alpha^2 + 2\alpha\beta}
        =   \frac{2(\alpha + \beta)^2}{\alpha + 2\beta}.
\]
Assume for contradiction that $1/\gamma_0(\bs{\alpha}) \in \Z$, and suppose $p$ is an odd prime that divides
$\alpha + 2\beta$.  Then $p$ divides $\alpha + \beta$, implying that $p$ divides $\beta$ and hence $\alpha$,
contradicting the fact that $\alpha$ and $\beta$ are coprime.  Therefore, it must be that
$\alpha + 2\beta = 2^k$ for some $k > 0$.  Note that $\alpha$ must be even and hence $\beta$ is odd.

If $\alpha/2$ is even, then $(\alpha + 2\beta)/2 = \alpha/2 + \beta$ is odd, implying $k = 1$ and
$\alpha + 2\beta = 2$.  This is not possible for positive integers $\alpha$ and $\beta$.  So suppose
$\alpha/2$ is odd, and then $\alpha/2 + \beta$ is even and divides $(\alpha + \beta)^2$.  Then $2$
divides $\alpha + \beta$, contradicting the parities of $\alpha$ and $\beta$.  Hence there are no
such $\alpha$ and $\beta$, completing the proof.
\end{proof}

It follows that we may restrict our attention to generic weight vectors.
We now observe the following immediate consequence of Corollary \ref{cor:ReductionExcludes}.

\begin{corollary}
\label{cor:Reduced3NotPairwiseCoprime}
Let $A = (a_1, a_2, a_3)$ be a generic weight vector associated to a linear $\Sp^1$-action on $\C^3$
such that each $a_i$ is nonzero and $\gcd(a_1, a_2, a_3) = 1$.
Assume $a_1 < 0$ and $a_2, a_3 > 0$.  If the associated reduced space
$M_0$ is $\Z$-graded regularly symplectomorphic to a linear
symplectic orbifold $\C^2/\Gamma$ for finite $\Gamma < \U_2$ then
either $\gcd(a_1, a_2) > 1$ or $\gcd(a_1, a_3) > 1$.
\end{corollary}

Note that if $a_1$, $a_2$, and $a_3$ are pairwise coprime, 
then the Laurent coefficient $\gamma_0(\bs{\alpha})$
is not the reciprocal of an integer and hence fails the Diophantine condition
\eqref{eq:Gamma0Diophantine}.  To see this, recall from Equation \eqref{eq:Gamma0Reduced} that
\[
    \frac{1}{\gamma_0(A)} = \frac{e_1 e_2 - e_3}{e_2}   =   e_1 - \frac{e_3}{e_2}
\]
where we recall that $e_i$ denotes the elementary symmetric polynomial in the $\alpha_j$ of degree $i$.
Note that the $e_i$ are obviously integer valued when the $\alpha_j \in \Z$.  Hence $1/\gamma_0(A)$ is an integer
if and only if $e_2$ divides $e_3$, i.e.
\[
    \frac{a_1 a_2 + a_1 a_3 + a_2 a_3}{a_1 a_2 a_3} \in \Z.
\]
As a consequence, considering prime divisors of the $a_j$, it is immediate that any prime divisor of a weight $a_j$ must divide
one of the other two weights.

We now establish the following, which will allow us in Subsection \ref{subsec:TypesU2} to exclude any
finite subgroup of $\U_2$ that does have exactly two quadratic invariants.

\begin{lemma}
\label{lem:Reduced3firstTaylor}
Let $A = (a_1, a_2, a_3)$ be a generic weight vector associated to a linear $\Sp^1$-action on $\C^3$
such that each $a_i$ is nonzero and $\gcd(a_1, a_2, a_3) = 1$.  Then the Taylor expansion of
$\Hilb_{A}^{\on}(x)$ at $x = 0$ begins
\[
    \Hilb_{A}^{\on}(x)  =   1 + 2x^2 + \cdots .
\]
\end{lemma}
\begin{proof}
Let $(z_1, z_2, z_3)$ denote the coordinates for $\C^3$ as above.
A polynomial in the $z_i, \cc{z_i}$ is invariant if and only if each monomial term is invariant,
so that the elements of $\R[V]^{\Sp^1}$ of degree $k$ are spanned by the invariant monomials of degree $k$.  Let
\[
    \Hilb_{A}^{\off}(x) =   b_0 + b_1 x + b_2 x^2 + \cdots
\]
denote the off-shell Hilbert series, i.e. the Hilbert series of $\R[\C^3]^{\Sp^1}$,
and then $b_k$ is the number of invariant monomials in $z_i, \cc{z_i}$ of
degree $k$.  Then $b_0 = 1$ counts the constant functions.  As each $a_i \neq 0$, there are clearly no
invariant linear monomials and hence $b_1 = 0$.  Finally, from the generic condition $\alpha_i \neq \alpha_j$ for
$i \neq j$, it is easy to see that the only invariant quadratic monomials are $z_i \cc{z_i}$ for $i = 1,2,3$ so
that $b_2 = 3$.  Then a simple computation using the fact that
\[
    (1 - x^2)\Hilb_{A}^{\off}(x)
    =
    \Hilb_{A}^{\on}(x)
\]
(see \cite[Proposition 2.1]{HerbigSeaton}) completes the proof.
\end{proof}
We note that using similar observations, it is easy to see that in fact,
\[
    \Hilb_{A}^{\on}(x)  =   1 + 2x^2 + c_3 x^3 + c_4 x^4 \cdots
\]
where $c_3$ is even and $c_4 \geq 3$ is odd, though we will only need Lemma \ref{lem:Reduced3firstTaylor} above.

Finally, we establish a restriction on the ratio of the first two nonzero Laurent series coefficients
of a reduced space that is $\Z$-graded regularly symplectomorphic to a linear orbifold.  For many finite subgroups
of $\U_2$ that we will meet in Subsection \ref{subsec:TypesU2}, the ratio of the first two nonzero Laurent coefficients
is $1$, $2$, or $\geq 3$, and hence they will be excluded using the following.

\begin{lemma}
\label{lem:Reduced3RatioG2G0}
Let $A = (a_1, a_2, a_3)$ be a generic weight vector associated to a linear $\Sp^1$-action on $\C^3$
such that each $a_i$ is nonzero and $\gcd(a_1, a_2, a_3) = 1$.  Assume that the associated reduced
space $M_0$ is $\Z$-graded regularly symplectomorphic to a linear symplectic orbifold $\C^2/\Gamma$
for finite $\Gamma < \U_2$.  Then the ratio $\gamma_0(A)/\gamma_2(A) < 3$ and is not equal to $1$ or $2$.
\end{lemma}
\begin{proof}
Recall the notation $\gcd\nolimits_{ij} = \gcd(\alpha_i, \alpha_j)$ from Subsection \ref{subsec:LaurentBack}.
Note that if $\gcd\nolimits_{ij} = \gcd\nolimits_{jk}$ for $i$, $j$, and $k$ distinct,
then this common divisor of all three weights must be $1$.  Let $r = \gamma_0(A)/\gamma_2(A)$.

$\bs{r < 3.}$
Applying Equations \eqref{eq:Gamma0Reduced} and \eqref{eq:Gamma23Reduced}
and solving for $e_2$, one obtains
\begin{equation}
\label{eq:Reduced3Ratio}
    e_2 =
    \frac{\gcd\nolimits_{12}^2 \alpha_3^2 + \gcd\nolimits_{13}^2 \alpha_2^2 + \gcd\nolimits_{23}^2 \alpha_1^2}
        {12/r + 2 - \gcd\nolimits_{12}^2 - \gcd\nolimits_{13}^2 - \gcd\nolimits_{23}^2}.
\end{equation}
As $e_2$ and the numerator of the right side are clearly positive, the denominator of the right side must be positive as well.
If $r \geq 3$, then this is only the case if $\gcd\nolimits_{12} = \gcd\nolimits_{13} = \gcd\nolimits_{23} = 1$.
However, this implies that the weights are pairwise relatively prime, which cannot be the case by
Corollary \ref{cor:Reduced3NotPairwiseCoprime}.  Hence $r < 3$.

$\bs{r \neq 1.}$
Suppose for contradiction that $r = 1$, and as $\gamma_0(A)$ does not depend on the order of the weights,
assume without loss of generality that
$\gcd\nolimits_{12}\geq\gcd\nolimits_{13}\geq\gcd\nolimits_{23}$.  Then for the denominator
of the right side of Equation \eqref{eq:Reduced3Ratio} to be positive, it must be that
$\gcd\nolimits_{13} = \gcd\nolimits_{23} = 1$ and $\gcd\nolimits_{12} \in \{ 1,2,3 \}$.
As above, we can exclude the case $\gcd\nolimits_{12} = 1$ by Corollary \ref{cor:Reduced3NotPairwiseCoprime}.

Suppose $\gcd\nolimits_{12} = 2$, and set $\alpha_1^\prime = \alpha_1/2$, $\alpha_2^\prime = \alpha_2/2$,
and $\alpha_3^\prime = \alpha_3$ so that the $\alpha_j^\prime$ are pairwise coprime.  Note that
$\alpha_3^\prime$ is odd, and at least one of $\alpha_1^\prime$ and $\alpha_2^\prime$ is odd as well.
Assume without loss of generality that $\alpha_2^\prime$ is odd.  Then Equation \eqref{eq:Reduced3Ratio}
can be rewritten as
\[
    4(2\alpha_1^\prime \alpha_2^\prime + \alpha_1^\prime \alpha_3^\prime + \alpha_2^\prime \alpha_3^\prime)
    =
    (\alpha_1^\prime)^2 + (\alpha_2^\prime)^2 + (\alpha_3^\prime)^2.
\]
If $\alpha_1^\prime$ is odd, then the right side is odd while the left side is divisible by $4$, a contradiction.
If $\alpha_1^\prime$ is even, then reducing mod $4$ yields
\[
    0   \equiv  (\alpha_2^\prime)^2 + (\alpha_3^\prime)^2\mod 4.
\]
By inspection, there are no odd integers that satisfy this congruence, allowing us to conclude that
$\gcd\nolimits_{12} \neq 2$.

So suppose $\gcd\nolimits_{12} = 3$, and then Equation \eqref{eq:Reduced3Ratio} becomes
\[
    3(\alpha_1 \alpha_2 + \alpha_1 \alpha_3 + \alpha_2 \alpha_3)
    =
    9\alpha_1^2 + \alpha_2^2 + \alpha_3^2.
\]
Considering the corresponding congruence mod $4$ and checking each of the $64$ possible values of the $\alpha_j$,
each of the $8$ solutions has $\alpha_1$, $\alpha_2$, and $\alpha_3$ even.  Hence there are no solutions
with $\gcd(\alpha_1, \alpha_2, \alpha_3) = 1$, and we conclude by contradiction that $r \neq 1$.

$\bs{r \neq 2.}$
Now suppose for contradiction that $r = 2$.  Assuming again with no loss of generality that
$\gcd\nolimits_{12}\geq\gcd\nolimits_{13}\geq\gcd\nolimits_{23}$, the denominator
of the right side of Equation \eqref{eq:Reduced3Ratio} is positive only when
$\gcd\nolimits_{13} = \gcd\nolimits_{23} = 1$ and $\gcd\nolimits_{12} \in \{ 1,2 \}$.
We again exclude $\gcd\nolimits_{12} = 1$ by Corollary \ref{cor:Reduced3NotPairwiseCoprime},
so that $\gcd\nolimits_{12} = 2$.

As above, set $\alpha_1^\prime = \alpha_1/2$, $\alpha_2^\prime = \alpha_2/2$,
and $\alpha_3^\prime = \alpha_3$.  Then the $\alpha_j^\prime$ are pairwise coprime,
$\alpha_3^\prime$ is odd, and we may assume without loss of generality that $\alpha_2^\prime$ is odd as well.
Rewrite Equation \eqref{eq:Reduced3Ratio} as
\[
    2\alpha_1^\prime \alpha_2^\prime + \alpha_1^\prime \alpha_3^\prime + \alpha_2^\prime \alpha_3^\prime
    =
    (\alpha_1^\prime)^2 + (\alpha_2^\prime)^2 + (\alpha_3^\prime)^2,
\]
and then for either possible parity of $\alpha_1^\prime$, the parities of the left and right sides of this equation
do not match.  Hence $r \neq 2$, completing the proof.
\end{proof}


\subsection{The Hilbert series of $\C^2/\Gamma$}
\label{subsec:HilbC2Gamma}

Let $\Gamma < \U_2$ be finite, and let $(w_1, w_2)$ denote the coordinates for $\C^2$.
We let $\omega_m$ denote a choice of primitive $m$th root of unity.  In addition, we let
$\Omega_m^{\operatorname{S}}$ denote the cyclic subgroup of $\SU_2$ of order $m$ generated by the element
$\diag(\omega_m, \omega_m^{-1})$ and $\Omega_m$ denote the cyclic subgroup of $\U_1 < \U_2$ of order $m$
generated by the scalar $\omega_m$.  Abusing notation in this manner, it will be convenient for us to identify
a scalar $\lambda$ with the element $\diag(\lambda,\lambda) \in \U_2$ that acts on $\C^2$ by scalar multiplication.

By the ADE-classification of finite subgroups of $\SU_2$, every cyclic subgroup of $\SU_2$ of order $m$
is conjugate to $\Omega_m^{\operatorname{S}}$.
The generator $\diag(\omega_m, \omega_m^{-1})$ of $\Omega_m^{\operatorname{S}}$ acts on $\C^2 \times \cc{\C^2}$
in coordinates $(w_1, w_2, \cc{w_1}, \cc{w_2})$ as $\diag(\omega_m, \omega_m^{-1}, \omega_m^{-1}, \omega_m)$.
By an application of Molien's formula \cite{Molien,SturmfelsBook}, one has
that all ten quadratic monomials are $\Omega_2^{\operatorname{S}}$-invariant, while if $m \geq 3$,
then the space of quadratic polynomials has dimension $4$; see \cite[Section 5.2.1]{FarHerSea} for details.
By inspection, the quadratic monomials
\begin{equation}
\label{eq:Finite2CyclicQuadraticInvars}
    w_1 \cc{w_1},
    \quad
    w_2\cc{w_2},
    \quad
    w_1 w_2,
    \quad\mbox{and}\quad
    \cc{w_1} \cc{w_2}
\end{equation}
are $\Omega_m^{\operatorname{S}}$-invariant and hence span the quadratic invariants when $m \geq 3$.

\begin{lemma}
\label{lem:Finite2QuadraticCyclics}
Let $\Gamma < \U_2$ be generated by $\Omega_m^{\operatorname{S}}$ and $\Omega_r$ for $m, r \geq 3$.
Then the quadratic invariants of $\Gamma$ are spanned by $w_1 \cc{w_1}$ and $w_2\cc{w_2}$.
\end{lemma}
\begin{proof}
With respect to the basis in Equation \eqref{eq:Finite2CyclicQuadraticInvars} for the quadratic
$\Omega_m^{\operatorname{S}}$-invariants, the action of $\omega_m$ is given by
$\diag(1, 1, \omega_m^2, \omega_m^{-2})$.  Applying the trace formula \cite[Lemma 2.2.2]{SturmfelsBook},
the dimension of the $\omega_m$-invariant subspace is given by
\[
    \frac{1}{m}\sum\limits_{k=0}^{m-1} \operatorname{Trace} \diag(1, 1, \omega_m^2, \omega_m^{-2})
    =
    2 + \frac{1}{m} \sum\limits_{k=0}^{m-1} \omega_m^{2k} + \omega_m^{-2k},
\]
where we note that if $m$ is even, then each term appears twice, but this is corrected for by the $1/m$
prefactor.  The remaining sum is over the $m$th (if $m$ is odd) or $m/2$nd (if $m$ is even) roots of unity
so that as $m \geq 3$, the sum vanishes.  It follows that the $\Gamma$-invariant quadratic polynomials have
dimension $2$, and hence by inspection are spanned by $w_1 \cc{w_1}$ and $w_2\cc{w_2}$.
\end{proof}

Let $\mathbb{D}_m$ denote the binary dihedral group, the subgroup of $\SU_2$ of order $4m$ generated by
$\diag(\omega_{2m}, \omega_{2m}^{-1})$ and
\begin{equation}
\label{eq:Finite2Defb}
    b = \begin{bmatrix}
            0   &   1   \\
            -1  &   0
        \end{bmatrix}.
\end{equation}
Note that $b$ acts on $\C^2 \times \cc{\C^2}$ as $\diag(b,b)$ (in $2\times 2$ blocks).
Again by the ADE-classification, any subgroup of $\SU_2$ isomorphic to a dihedral group
is conjugate to some $\mathbb{D}_m$.  As above, by Molien's formula and computed in \cite[Section 5.2.2]{FarHerSea},
the quadratic invariants of $\mathbb{D}_1 = \langle b \rangle$ are of dimension $4$, and the quadratic invariants of
$\mathbb{D}_m$ for $m > 1$ are of dimension $1$.  One checks that
\[
    w_1 \cc{w_1} + w_2\cc{w_2},
    \quad
    w_1 \cc{w_2} - w_2\cc{w_1},
    \quad
    w_1^2 +  w_2^2,
    \quad\mbox{and}\quad
    \cc{w_1}^2 +  \cc{w_2}^2
\]
are $b$-invariant and hence span the quadratic $\mathbb{D}_1$-invariants, while
$w_1 \cc{w_1} + w_2\cc{w_2}$ spans the $\mathbb{D}_m$ invariants for $m > 1$.

Note that if $\mathbb{D}_m < \Gamma$, then any $\Gamma$-invariant polynomial is obviously $\mathbb{D}_m$-invariant.
Combining these observations yields the following, which along with Lemma \ref{lem:Reduced3firstTaylor} will allow
us to exclude any finite subgroup of $\U_2$ containing $\mathbb{D}_m$ for $m \geq 2$.

\begin{lemma}
\label{lem:Finite2QuadraticDihedral}
Suppose the finite group $\Gamma < \U_2$ contains the binary dihedral group $\mathbb{D}_m$ for $m > 1$.
Then the dimension of the quadratic elements of $\R[\C^2]^\Gamma$ is at most $1$.
\end{lemma}

It will also be helpful for us to note the following computation, which will be used extensively.

\begin{lemma}
\label{lem:Finite2lambdaBEigenvals}
For any $\lambda\in\C$, $\lambda b^k$ is a pseudoreflection if and only if $k$ is odd and $\lambda = \pm \sqrt{-1}$,
in which case it is a pseudoreflection of order $2$.
\end{lemma}
\begin{proof}
First note that $\lambda b^2 = -\lambda$ and $\lambda b^4 = \lambda$ each have two eigenvalues of multiplicity $2$ and hence
are never pseudoreflections.  A simple computation demonstrates that the eigenvalues of $\lambda b$ and $\lambda b^3 = - \lambda b$ are
$\sqrt{-1} \lambda$ and $-\sqrt{-1}\lambda$ from which the result follows.
\end{proof}


\subsection{Elimination of each $\Gamma < \U_2$}
\label{subsec:TypesU2}

First, let us briefly recall the classification of finite subgroups of $\U_2$ of \cite{DuVal,CoxeterBook}.
Using the surjective $2$-to-$1$ group homomorphism $\U_1 \times \SU_2 \to \U_2$ given by multiplication
$(l, r) \mapsto lr$ with kernel $\pm(1, \id)$, every element of $\U_2$ can be expressed as the image of
$(l, r) \in \U_1\times\SU_2$, and this expression is unique up to $(-l, -r) = (l,r)$.  Hence, any finite
subgroup of $\U_2$ is of the form
\[
    (L/L_K; R/R_K)_\phi
    =
    \{ (l, r) \in L \times R : \phi(lL_K) = rR_K \}
\]
where $L_K \unlhd L$ are finite subgroups of $\U_1$ and $R_K \unlhd R$
are finite subgroups of $\SU_2$ such that $L/L_K$ is isomorphic to $R/R_K$, and
$\phi\co L/L_K\to R/R_K$ is a choice of isomorphism.  If $\phi$ is unique, we write
simply $(L/L_K; R/R_K)$.  The group $(L/L_K; R/R_K)_\phi$ has order $|R| |L_K|/2$.
Note that $(L/L_K; R/R_K)_\phi$ contains $L_K$ and $R_K$ as subgroups.

Let $\mathbb{T}_{24}$, $\mathbb{O}_{48}$, and $\mathbb{I}_{120}$ denote
the binary tetrahedral, octahedral, and icosahedral groups, respectively.
As above, we let $\Omega_m < \U_1$ and $\Omega_m^{\operatorname{S}} < \SU_2$ denote the cyclic
subgroups of order $m$ and let $\mathbb{D}_m < \SU_2$ denote the binary dihedral group of order $4m$.
Below, we recall each of the nine types of finite subgroups of $\U_2$ and demonstrate that the associated
linear symplectic orbifolds are not $\Z$-graded regularly symplectomorphic to
an $\Sp^1$-reduced space corresponding to weight vector $A = (-\alpha_1, \alpha_2, \alpha_3)$
with each $\alpha_i > 0$.


\subsubsection*{Type I}
\label{subsubsec:TypeI}

A Type I group $\Gamma$ is of the form $(\Omega_{2m}/\Omega_f; \Omega_{2n}^{\operatorname{S}}/\Omega_g^{\operatorname{S}})_d$
where $f\equiv g \mod 2$, and $d$ is relatively prime to $2m/f = 2n/g$ and indicates the isomorphism
$\Omega_{2m}/\Omega_f \to \Omega_{2n}^{\operatorname{S}} /\Omega_g^{\operatorname{S}}$
sending the class of 1 to the class of $d$.  Every element of $\Gamma$ is of the form
$\omega_{2m}^j\diag(\omega_{2n}^k,\omega_{2n}^{-k})$ for some $j$ and $k$ so that $\Gamma$
is abelian and consists of diagonal elements.  Clearly, the only pseudoreflections $\Gamma$ can contain
are of the form $\diag(1,\lambda)$ or $\diag(\lambda,1)$ for a scalar $\omega$ so that the number of
primitive pseudoreflections in $\Gamma$ is either $0$, $1$, or $2$ (see Subsection \ref{subsec:LaurentBack}).
By Corollary \ref{cor:Gamma2Positive}, $\Gamma$ must contain at least one pseudoreflection.  We can exclude
the other Type I groups with the following two results.

\begin{lemma}
\label{lem:TypeI2pseudoreflections}
Let $A = (a_1, a_2, a_3)$ be a generic weight vector associated to a linear $\Sp^1$-action on $\C^3$
such that each $a_i$ is nonzero and $\gcd(a_1, a_2, a_3) = 1$.  Then the associated reduced space $M_0$
cannot be $\Z$-graded regularly symplectomorphic to $\C^2/\Gamma$ where $\Gamma < \U_2$ is abelian and contains
two primitive pseudoreflections.
\end{lemma}
\begin{proof}
Assume such a $\Z$-graded regular symplectomorphism does exist for weight vector $A = (-\alpha_1, \alpha_2, \alpha_3)$
and $\Gamma < \U_2$, and fix a set of two primitive pseudoreflections of $\Gamma$. Then $\C^2$ has two complex codimension-$1$ orbit
type strata implying that $M_0$ does as well. It follows from the descriptions of the orbit type strata of $M_0$
in Section \ref{sec:n3Red} that $\gcd\nolimits_{12} > 1$ and $\gcd\nolimits_{13} > 1$, and the closures of the corresponding
orbit type strata are given by the reduced spaces associated to $(-\alpha_1, \alpha_2)$ and $(-\alpha_1, \alpha_3)$, respectively.
By \cite[Theorem 7]{FarHerSea}, these two orbit type strata are $\Z$-graded regularly symplectomorphic to $\C/\Z_{N_1}$ and $\C/\Z_{N_2}$,
respectively, where $N_1 = (\alpha_1 + \alpha_2)/\gcd\nolimits_{12}$ and $N_2 = (\alpha_1 + \alpha_3)/\gcd\nolimits_{13}$.
By Proposition \ref{prop:ReductionStep}, the $\Z$-graded regular symplectomorphism from $M_0$ to $\C^2/\Gamma$
restricts to a $\Z$-graded regular symplectomorphism on each of the closures of the orbit type strata.
It follows that the closures of the $2$-dimensional orbit type strata of $\C^2/\Gamma$ are isomorphic to
$\C/\Z_{N_1}$ and $\C/\Z_{N_2}$, respectively.

Note that if $h \in \Gamma$ is a primitive pseudoreflection, then the closure of the orbit type
$\langle h \rangle$ is given by $(\C^2)^h = \C$ equipped with the restricted action of $\Gamma$.  As $\Gamma$ is abelian,
the effectivized action is that of $\Gamma/\langle h \rangle$, which acts on $\C$ as an element of $\U_1$ and hence is cyclic.  The
$\Z$-graded invariants of a cyclic group acting on $\C$  determine the order of the group (see \cite[page 21]{FarHerSea}),
so we have that $\Gamma/\langle h \rangle$ is isomorphic to either
$\C/\Z_{N_1}$ or $\C/\Z_{N_2}$.  If $r$ and $s$ denote the orders of the two primitive pseudoreflections in $\Gamma$,
it then follows that up to relabeling, $|\Gamma|/r = N_1$ and $|\Gamma|/s = N_2$.

Using Equations \eqref{eq:Gamma0Reduced} and \eqref{eq:Gamma0123Finite} and the fact that $\gamma_0(A) = \gamma_0(\Gamma)$,
we have
\[
    |\Gamma| = \frac{1}{\gamma_0(A)}
    =
    \frac{(\alpha_1 + \alpha_2)(\alpha_1 + \alpha_3)(\alpha_2 + \alpha_3)}
    {\alpha_1 \alpha_2 + \alpha_1 \alpha_3 + \alpha_2 \alpha_3},
\]
and the equation $N_1 = |\Gamma|/r = (\alpha_1 + \alpha_2)/\gcd\nolimits_{12}$ becomes
\[
    \frac{\alpha_1 + \alpha_2}{\gcd\nolimits_{12}}
    =
    \frac{(\alpha_1 + \alpha_2)(\alpha_1 + \alpha_3)(\alpha_2 + \alpha_3)}
    {r(\alpha_1 \alpha_2 + \alpha_1 \alpha_3 + \alpha_2 \alpha_3)}.
\]
Solving for $r$ yields
\begin{equation}
\label{eq:TypeI2pseudoreflectionsR}
    r   =       \frac{(\alpha_1 + \alpha_3)(\alpha_2 + \alpha_3)\gcd\nolimits_{12}}
                {\alpha_1 \alpha_2 + \alpha_1 \alpha_3 + \alpha_2 \alpha_3}
        =       \gcd\nolimits_{12} + \frac{\alpha_3^2\gcd\nolimits_{12}}
                {\alpha_1 \alpha_2 + \alpha_1 \alpha_3 + \alpha_2 \alpha_3}.
\end{equation}
Similarly from $N_2 = |\Gamma|/s = (\alpha_1 + \alpha_3)/\gcd\nolimits_{13}$, we have
\begin{equation}
\label{eq:TypeI2pseudoreflectionsS}
    s   =       \frac{(\alpha_1 + \alpha_2)(\alpha_2 + \alpha_3)\gcd\nolimits_{13}}
                {\alpha_1 \alpha_2 + \alpha_1 \alpha_3 + \alpha_2 \alpha_3}
        =       \gcd\nolimits_{13} + \frac{\alpha_2^2 \gcd\nolimits_{13}}
                {\alpha_1 \alpha_2 + \alpha_1 \alpha_3 + \alpha_2 \alpha_3}.
\end{equation}
To derive the contradiction, we will show that it is not possible for the expressions in
Equations \eqref{eq:TypeI2pseudoreflectionsR} and \eqref{eq:TypeI2pseudoreflectionsS} to both be integers
when $\gcd(\alpha_1,\alpha_2,\alpha_3) = 1$.

Define
\[
    \alpha_1^\prime
        =  \frac{\alpha_1}{\gcd\nolimits_{12}\gcd\nolimits_{13}},
    \quad
    \alpha_2^\prime
        =  \frac{\alpha_2}{\gcd\nolimits_{12}\gcd\nolimits_{23}},
    \quad\mbox{and}\quad
    \alpha_3^\prime
        =  \frac{\alpha_3}{\gcd\nolimits_{13}\gcd\nolimits_{23}}.
\]
As $\gcd(\alpha_1,\alpha_2,\alpha_3) = 1$, it follows that $\gcd\nolimits_{12}$ and $\gcd\nolimits_{13}$ are
coprime so that $\alpha_1^\prime$ is an integer (and similarly for $\alpha_2^\prime$ and $\alpha_3^\prime$).  Moreover, it is clear that
$\alpha_1^\prime$, $\alpha_2^\prime$, and $\alpha_3^\prime$ are pairwise coprime.  Rewriting the last term in
Equation \eqref{eq:TypeI2pseudoreflectionsR} in terms of these values yields
\begin{equation}
\label{eq:TypeI2pseudoreflectionsR2}
    \frac{\alpha_3^2\gcd\nolimits_{12}}
        {\alpha_1 \alpha_2 + \alpha_1 \alpha_3 + \alpha_2 \alpha_3}
    =
    \frac{(\alpha_3^\prime)^2\gcd\nolimits_{13}\gcd\nolimits_{23}}
        {\alpha_1^\prime \alpha_2^\prime \gcd\nolimits_{12} + \alpha_1^\prime \alpha_3^\prime \gcd\nolimits_{13}
        + \alpha_2^\prime \alpha_3^\prime \gcd\nolimits_{13}},
\end{equation}
and the last term in Equation \eqref{eq:TypeI2pseudoreflectionsS} becomes
\begin{equation}
\label{eq:TypeI2pseudoreflectionsS2}
    \frac{\alpha_2^2 \gcd\nolimits_{13}}{\alpha_1 \alpha_2 + \alpha_1 \alpha_3 + \alpha_2 \alpha_3}
    =
    \frac{(\alpha_2^\prime)^2\gcd\nolimits_{12}\gcd\nolimits_{23}}
        {\alpha_1^\prime \alpha_2^\prime \gcd\nolimits_{12} + \alpha_1^\prime \alpha_3^\prime \gcd\nolimits_{13}
        + \alpha_2^\prime \alpha_3^\prime \gcd\nolimits_{23}}.
\end{equation}
Define $\kappa = \alpha_1^\prime \alpha_2^\prime \gcd\nolimits_{12} + \alpha_1^\prime \alpha_3^\prime \gcd\nolimits_{13}
+ \alpha_2^\prime \alpha_3^\prime \gcd\nolimits_{23}$, let $p$ be a prime that divides $\kappa$, and refer first to Equation
\eqref{eq:TypeI2pseudoreflectionsR2}.  Then $p$ divides $(\alpha_3^\prime)^2\gcd\nolimits_{13}\gcd\nolimits_{23}$.
If $p$ divides $\alpha_3^\prime$, then $p$ must divide
$\alpha_1^\prime \alpha_2^\prime \gcd\nolimits_{12}$.  However, as $\alpha_1^\prime$, $\alpha_2^\prime$, and $\alpha_3^\prime$
are pairwise coprime, it must be that $p$ divides $\gcd\nolimits_{12}$.  But then $p$ divides $\alpha_1$, $\alpha_2$, and $\alpha_3$,
a contradiction, so it must be the case that $\kappa$ and $\alpha_3^\prime$ are coprime.  Therefore, $\kappa$ divides
$\gcd\nolimits_{13}\gcd\nolimits_{23}$.

Applying an identical argument with Equation \eqref{eq:TypeI2pseudoreflectionsS2} demonstrates that
$\kappa$ and $\alpha_2^\prime$ are coprime so that $\kappa$ divides $\gcd\nolimits_{12}\gcd\nolimits_{23}$.  However,
if $q$ is any prime that divides $\kappa$, then $q$ cannot divide both
$\gcd\nolimits_{12}$ and $\gcd\nolimits_{13}$, for then it would divide all three weights, so it must divide
$\gcd\nolimits_{23}$.  Therefore, $\kappa$
must in fact be coprime to both $\gcd\nolimits_{12}$ and $\gcd\nolimits_{13}$, and hence
$\kappa$ divides $\gcd\nolimits_{23}$.  However, $\kappa > \gcd\nolimits_{23}$, yielding a contradiction
and completing the proof.
\end{proof}

Similarly, we have the following.

\begin{lemma}
\label{lem:TypeI1pseudoreflection}
Let $A = (a_1, a_2, a_3)$ be a generic weight vector associated to a linear $\Sp^1$-action on $\C^3$
such that each $a_i$ is nonzero and $\gcd(a_1, a_2, a_3) = 1$.  Then the associated reduced space $M_0$
cannot be $\Z$-graded regularly symplectomorphic to $\C^2/\Gamma$ where $\Gamma$ is abelian and contains
one primitive pseudoreflection.
\end{lemma}
\begin{proof}
Assume such a $\Z$-graded regular symplectomorphism does exist for weight vector $A = (-\alpha_1, \alpha_2, \alpha_3)$
and $\Gamma < \U_2$ with one fixed primitive pseudoreflection of order $r$.  Following the proof of Lemma \ref{lem:TypeI2pseudoreflections},
we conclude that $M_0$ and $\C^2/\Gamma$ each have one complex codimension-$1$ orbit type stratum.  We may assume without
loss of generality that $\gcd\nolimits_{12} = 1$ and $\gcd\nolimits_{13} > 1$, and then the closure of the complex codimension-$1$
orbit type stratum of $M_0$ is $\Z$-graded regularly symplectomorphic to $\C/\Z_N$ with $N = (\alpha_1 + \alpha_3)/\gcd\nolimits_{13}$.
Then $|\Gamma|/r = N$, implying that
\[
    r   =   \frac{(\alpha_1 + \alpha_2)(\alpha_2 + \alpha_3)\gcd\nolimits_{13}}{\alpha_1 \alpha_2 + \alpha_1 \alpha_3 + \alpha_2 \alpha_3}
        =   \gcd\nolimits_{13} + \frac{\alpha_2^2\gcd\nolimits_{13}}
            {\alpha_1 \alpha_2 + \alpha_1 \alpha_3 + \alpha_2 \alpha_3} \in \Z.
\]
Expressing this in terms of
\[
    \alpha_1^\prime
        =  \frac{\alpha_1}{\gcd\nolimits_{13}},
        \quad
    \alpha_2^\prime
        =  \frac{\alpha_2}{\gcd\nolimits_{23}}
        \quad\mbox{and}\quad
    \alpha_3^\prime
        =  \frac{\alpha_3}{\gcd\nolimits_{13}\gcd\nolimits_{23}},
\]
we have
\[
    \frac{\alpha_2^2\gcd\nolimits_{13}}{\alpha_1\alpha_2 + \alpha_1\alpha_3 + \alpha_2\alpha_3}
    =
    \frac{(\alpha_2^\prime)^2 \gcd\nolimits_{23}}
        {\alpha_1^\prime \alpha_2^\prime  + \alpha_1^\prime \alpha_3^\prime \gcd\nolimits_{13}
        + \alpha_2^\prime \alpha_3^\prime \gcd\nolimits_{23}}.
\]
Let
$\kappa = \alpha_1^\prime \alpha_2^\prime  + \alpha_1^\prime \alpha_3^\prime \gcd\nolimits_{13} + \alpha_2^\prime \alpha_3^\prime \gcd\nolimits_{23}$,
and let $p$ be a prime that divides $\kappa$.  If $p$ divides $\alpha_2^\prime$, then it divides
$\alpha_1^\prime \alpha_3^\prime \gcd\nolimits_{13}$.  However, as $\alpha_1^\prime$, $\alpha_2^\prime$, and $\alpha_3^\prime$ are pairwise
coprime, $p$ must divide $\gcd\nolimits_{13}$.  But then $p$ divides $\alpha_1$, $\alpha_2$, and $\alpha_3$, a contradiction.
Therefore, it must be that $\kappa$ is coprime to $\alpha_2^\prime$, and hence $\kappa$ must divide
$\gcd\nolimits_{23}$.  But $\kappa > \gcd\nolimits_{23}$, a contradiction.
\end{proof}

As a consequence of Lemmas \ref{lem:TypeI2pseudoreflections} and \ref{lem:TypeI1pseudoreflection}, $\Gamma$ cannot be a Type I group.


\subsubsection*{Type II}
\label{subsubsec:TypeII}

Suppose $\Gamma$ is a Type II group, a group of the form $(\Omega_{2m}/\Omega_{2m}; \mathbb{D}_\ell/ \mathbb{D}_\ell)$.
By Lemma \ref{lem:Reduced3firstTaylor}, the degree $2$
coefficient of the Taylor expansion of $\Hilb_A^{\on}(x)$ at $x = 0$ is equal to $2$.  Then
by Lemma \ref{lem:Finite2QuadraticDihedral}, it must be that $\ell = 1$ so that
$\Gamma = (\Omega_{2m}/\Omega_{2m}; \mathbb{D}_1/ \mathbb{D}_1)$ for some $m$.

The $4m$ elements of $\Gamma$ are of the form $\omega_{2m}^j b^k$ where $b$ is as in Equation \eqref{eq:Finite2Defb};
note that $\omega_{2m}^j b^k = \omega_{2m}^{j+m} b^{k+2}$.  By Lemma \ref{lem:Finite2lambdaBEigenvals}, $\Gamma$ contains
pseudoreflections if and only if $m$ is even, in which case it contains the
two pseudoreflections $\sqrt{-1}b = -\sqrt{-1}b^3$ and $-\sqrt{-1}b = \sqrt{-1}b^3$, each of order $2$.  By Corollary \ref{cor:Gamma2Positive},
it must be that $m$ is even, and it is easy to see that $\{ \sqrt{-1}b, -\sqrt{-1}b \}$ is a set of primitive pseudoreflections.
Applying Equation \eqref{eq:Gamma0123Finite}, the first two Laurent coefficients of $\Hilb_{\R[\C^2]^\Gamma|\R}(x)$ are
\[
    \gamma_0(\Gamma)    =   \frac{1}{4m}
    \quad\mbox{and}\quad
    \gamma_2(\Gamma)    =   \frac{1}{8m}.
\]
But then $\gamma_0(\Gamma)/\gamma_2(\Gamma) = 2$, and by Lemma \ref{lem:Reduced3RatioG2G0}, $\Hilb_{\R[\C^2]^\Gamma|\R}(x)$
cannot coincide with $\Hilb_A^{\on}(x)$.  Hence $\Gamma$ cannot be of Type II.


\subsubsection*{Type III}
\label{subsubsec:TypeIII}

A Type III group is of the form $(\Omega_{4m}/\Omega_{2m}; \mathbb{D}_\ell/ \Omega_{2\ell}^{\operatorname{S}})$.
Suppose $\Gamma$ is such a group, and first assume $m > 1$ and $\ell > 1$.  Then $\Gamma$ contains the subgroups
$\Omega_{2m}$ as well as $\Omega_{2\ell}^{\operatorname{S}}$, and the common quadratic invariants of these
groups are $w_1 \cc{w_1}$ and $w_2 \cc{w_2}$ by Lemma \ref{lem:Finite2QuadraticCyclics}.  However, $\Gamma$
also contains the element $\omega_{4m}b$, which acts on $\C^2 \times \cc{\C^2}$ in coordinates
$(w_1, w_2, \cc{w_1}, \cc{w_2})$ as $\diag(\omega_{4m} b, \omega_{4m}^{-1} b)$ (in $2\times 2$ blocks).
A simple computation demonstrates that the action of this element permutes $w_1 \cc{w_1}$ and $w_2 \cc{w_2}$
so that the quadratic invariants of $\Gamma$ are of dimension at most $1$.  By Lemma \ref{lem:Reduced3firstTaylor},
$\Hilb_{\R[\C^2]^\Gamma|\R}(x) \neq \Hilb_A^{\on}(x)$.

Now suppose $\ell = 1$ so that $\Gamma = (\Omega_{4m}/\Omega_{2m}; \mathbb{D}_1/ \Omega_2^{\operatorname{S}})$
consists of the $4m$ elements $\omega_{2m}^j$ for $0 \leq j \leq 2m - 1$ and $\omega_{4m}^k b$
for $k$ odd, $1 \leq k \leq 4m - 1$.  The elements $\omega_{2m}^j$ are clearly not pseudoreflections, and
$\omega_{4m}^k b$ is a pseudoreflection if and only if $\omega_{4m}^k = \pm\sqrt{-1}$ by
Lemma \ref{lem:Finite2lambdaBEigenvals}.  If $m$ is even, then as $k$ must be odd, $\Gamma$ contains no
pseudoreflections and can be ruled out by Corollary \ref{cor:Gamma2Positive}.  If $m$ is odd, then $\Gamma$ contains two pseudoreflections
of order $2$ which form a set of primitive pseudoreflections.  Again by Equation \eqref{eq:Gamma0123Finite},
\[
    \gamma_0(\Gamma)    =   \frac{1}{4m}
    \quad\mbox{and}\quad
    \gamma_2(\Gamma)    =   \frac{1}{8m}
\]
so that $\gamma_0(\Gamma)/\gamma_2(\Gamma) = 2$.  By Lemma \ref{lem:Reduced3RatioG2G0}, $\Hilb_{\R[\C^2]^\Gamma|\R}(x)$
cannot coincide with $\Hilb_A^{\on}(x)$.

Finally, suppose $\ell > 1$ and $m = 1$.  Then
$\Gamma = (\Omega_4/\Omega_2; \mathbb{D}_\ell/ \Omega_{2\ell}^{\operatorname{S}})$ contains the $4\ell$ elements
$\diag(\omega_{2\ell},\omega_{2\ell}^{-1})^j$ and $\sqrt{-1}b\diag(\omega_{2\ell},\omega_{2\ell}^{-1})^j$
for $0 \leq j \leq 2\ell - 1$.  A simple computation using this description of the elements of $\Gamma$ demonstrates
that the three polynomials $w_1 w_2$, $\cc{w_1} \cc{w_2}$,
and $w_1 \cc{w_1} + w_2 \cc{w_2}$ are always $\Gamma$-invariant, implying that the $x^2$-coefficient in
the Taylor expansion of $\Hilb_{\R[\C^2]^\Gamma|\R}(x)$ at $x = 0$ is at least $3$.
As the $x^2$-coefficient of the Taylor expansion of $\Hilb_A^{\on}(x)$ is equal to $2$ by Lemma \ref{lem:Reduced3firstTaylor},
$\Gamma$ cannot be a Type III group.

Though we will not require the following, we note that the above coefficient is in fact equal to $3$,
which can be seen by computing the Hilbert series explicitly using Molien's formula and \cite[Theorem 4.2]{Gessel} to yield
\[
    \Hilb_{\R[\C^2]^\Gamma|\R}(x)
    =
    \frac{1 + (2\ell - 1) x^{2\ell} - (2\ell - 1) x^{2\ell+2} - x^{4\ell+2}}
        {(1 - x^2)^3(1 - x^{2\ell})^2}.
\]
Then a simple computation demonstrates that as $\ell > 1$, the $x^2$-coefficient of the Taylor series of $\Hilb_{\R[\C^2]^\Gamma|\R}(x)$
is equal to $3$.


\subsubsection*{Type III$^\prime$}
\label{subsubsec:TypeIIIprime}

Suppose $\Gamma$ is a Type III$^\prime$ group of the form
$(\Omega_{4m}/\Omega_m; \mathbb{D}_\ell/ \Omega_\ell^{\operatorname{S}})$ for $m$ and $l$ odd.
If $m > 1$ and $\ell > 1$, then just as in the case of a Type III group above,
$\Gamma$ contains the subgroups $\Omega_m$ and $\Omega_\ell^{\operatorname{S}}$
as well as the element $\omega_{4m}b$.  Therefore, the quadratic invariants of $\Gamma$
are of dimension $0$ or $1$, and $\Hilb_{\R[\C^2]^\Gamma|\R}(x) \neq \Hilb_A^{\on}(x)$.

Assume $\ell = 1$, and then the $2m$ elements of $\Gamma = (\Omega_{4m}/\Omega_m; \mathbb{D}_1/1)$
are $\omega_{4m}^{4j}$ and $\omega_{4m}^{4j+1}b$ for $0 \leq j \leq m - 1$.
The scalars $\omega_{4m}^{4j}$ are never pseudoreflections, while $\omega_{4m}^{4j+1}b$
is a pseudoreflection if and only if $\omega_{4m}^{4j+1} = \pm \sqrt{-1}$ by Lemma \ref{lem:Finite2lambdaBEigenvals}.
Recalling that $m$ is odd, this occurs exactly once; if $m \equiv 1\mod 4$, then it occurs when $4j+1 = m$, and if
$m \equiv 3\mod 4$, it occurs when $4j+1 = 3m$.  In either case, $\Gamma$ contains exactly one pseudoreflection of order $2$.
By Equation \eqref{eq:Gamma0123Finite},
\[
    \gamma_0(\Gamma)    =   \frac{1}{2m}
    \quad\mbox{and}\quad
    \gamma_2(\Gamma)    =   \frac{1}{8m}.
\]
Then $\gamma_0(\Gamma)/\gamma_2(\Gamma) = 4$ so that $\Hilb_{\R[\C^2]^\Gamma|\R}(x) \neq \Hilb_A^{\on}(x)$
by Lemma \ref{lem:Reduced3RatioG2G0}.

Finally, assume $m  = 1$ and $\ell > 1$ so that $\Gamma = (\Omega_4/1; \mathbb{D}_\ell/ \Omega_\ell^{\operatorname{S}})$
consists of the $2\ell$ elements $\diag(\omega_\ell,\omega_\ell^{-1})^j$ and $\sqrt{-1}b\diag(\omega_\ell,\omega_\ell^{-1})^j$
for $0 \leq j \leq \ell - 1$.  Just as in the case of Type III groups (with $\ell > 1$ and $m = 1$),
one checks that $w_1 w_2$, $\cc{w_1} \cc{w_2}$, and $w_1 \cc{w_1} + w_2 \cc{w_2}$ are always $\Gamma$-invariant
so that the $x^2$-coefficient in the Taylor series of $\Hilb_{\R[\C^2]^\Gamma|\R}(x)$ is at least $3$.
By Lemma \ref{lem:Reduced3firstTaylor}, $\Gamma$ cannot be a Type III$^\prime$ group.

Still using the same methods as in Type III, $\Hilb_{\R[\C^2]^\Gamma|\R}(x)$ can be computed explicitly to be
\[
    \Hilb_{\R[\C^2]^\Gamma|\R}(x)
    =
    \frac{1 + (\ell - 1) x^{\ell} - (\ell - 1) x^{2\ell+2} - x^{2\ell+2}}
        {(1 - x^2)^3(1 - x^{\ell})^2},
\]
and then it is easy to verify that the $x^2$-coefficient of the Taylor series is equal to $3$ (as $\ell$ must be odd and so
$\ell \geq 3$).


\subsubsection*{Type IV}
\label{subsubsec:TypeIV}

Suppose $\Gamma$ is a Type IV group of the form $(\Omega_{4m}/\Omega_{2m}; \mathbb{D}_{2\ell}/ \mathbb{D}_\ell)$.
By Lemmas \ref{lem:Reduced3firstTaylor} and \ref{lem:Finite2QuadraticDihedral}, if the $x^2$-coefficients
in the Taylor expansions of $\Hilb_{\R[\C^2]^\Gamma|\R}(x)$ and $\Hilb_A^{\on}(x)$ coincide, it must be
that $\ell = 1$ and then $\Gamma = (\Omega_{4m}/\Omega_{2m}; \mathbb{D}_2/ \mathbb{D}_1)$.  Note that $\Gamma$
contains the Type II group $(\Omega_{2m}/\Omega_{2m}; \mathbb{D}_1/ \mathbb{D}_1)$ as an index $2$ subgroup,
and the nontrivial element of the quotient $\Gamma/(\Omega_{2m}/\Omega_{2m}; \mathbb{D}_1/ \mathbb{D}_1)$ is
the coset of $\omega_{4m}\diag(\sqrt{-1},-\sqrt{-1})$.  We consider two cases.

If $m$ is odd, then $(\Omega_{2m}/\Omega_{2m}; \mathbb{D}_1/ \mathbb{D}_1)$ contains no pseudoreflections (see Type II above).
Hence any pseudoreflection in $\Gamma$ must be of the form
\begin{equation}
\label{eq:TypeIV.1}
    \omega_{4m}\diag(\sqrt{-1},-\sqrt{-1})\omega_{2m}^j
    =
    \omega_{4m}^{2j+1}\diag(\sqrt{-1},-\sqrt{-1})
\end{equation}
with eigenvalues $\pm \sqrt{-1}\omega_{4m}^{2j+1}$ or
\begin{equation}
\label{eq:TypeIV.2}
    \omega_{4m}\diag(\sqrt{-1},-\sqrt{-1})\omega_{2m}^j b
    =
    \omega_{4m}^{2j+1}
    \begin{bmatrix} 0   &   i   \\
                    i   &   0
    \end{bmatrix},
\end{equation}
also with eigenvalues $\pm \sqrt{-1}\omega_{4m}^{2j+1}$.  Then it is easy to see that $\Gamma$ contains
exactly four pseudoreflections,
\begin{align*}
    \omega_{4m}^m \diag(\sqrt{-1},-\sqrt{-1})
    &=
    \diag(-1,1),
    \quad
    &\omega_{4m}^m
    \begin{bmatrix} 0   &   i   \\
                    i   &   0
    \end{bmatrix}
    &=
    \begin{bmatrix} 0   &  -1   \\
                    -1  &   0
    \end{bmatrix},
    \\
    \omega_{4m}^{3m} \diag(\sqrt{-1},-\sqrt{-1})
    &=
    \diag(1,-1),
    \quad\mbox{and}
    &\omega_{4m}^{3m}
    \begin{bmatrix} 0   &   i   \\
                    i   &   0
    \end{bmatrix}
    &=
    \begin{bmatrix} 0   &   1   \\
                    1   &   0
    \end{bmatrix},
\end{align*}
each of order $2$.  By Equation \eqref{eq:Gamma0123Finite},
\[
    \gamma_0(\Gamma)
    =
    \gamma_2(\Gamma)
    =
    \frac{1}{8m}
\]
so that $\gamma_0(\Gamma)/\gamma_2(\Gamma) = 1$, and $\Hilb_{\R[\C^2]^\Gamma|\R}(x) \neq \Hilb_A^{\on}(x)$
by Lemma \ref{lem:Reduced3RatioG2G0}.

Now suppose $m$ is even, and then $(\Omega_{2m}/\Omega_{2m}; \mathbb{D}_1/ \mathbb{D}_1)$ contains the two
pseudoreflections $\pm\sqrt{-1}b$ of order $2$ (see Type II).  As in the case of $m$ odd, the other elements
of $\Gamma$ are of the forms described in Equations \eqref{eq:TypeIV.1} and \eqref{eq:TypeIV.2} with eigenvalues
$\pm \sqrt{-1}\omega_{4m}^{2j+1}$.  As $m$ is even, $\pm \sqrt{-1}\omega_{4m}^{2j+1} \neq 1$ for each $j$ so that there
are no additional pseudoreflections.  Equation \eqref{eq:Gamma0123Finite} yields
\[
    \gamma_0(\Gamma)    =   \frac{1}{8m}
    \quad\mbox{and}\quad
    \gamma_2(\Gamma)    =   \frac{1}{16m}
\]
so that $\gamma_0(\Gamma)/\gamma_2(\Gamma) = 2$, and $\Hilb_{\R[\C^2]^\Gamma|\R}(x) \neq \Hilb_A^{\on}(x)$
by Lemma \ref{lem:Reduced3RatioG2G0}.  Therefore, $\Gamma$ cannot be a Type IV group.


\subsubsection*{Types V through IX}
\label{subsubsec:TypeVtoIX}

The remaining types of finite subgroups of $\U_2$ are Type V of the form
$(\Omega_{2m}/\Omega_{2m}; \mathbb{T}_{24}/\mathbb{T}_{24})$,
Type VI of the form
$(\Omega_{6m}/\Omega_{2m}; \mathbb{T}_{24}/\mathbb{D}_2)$,
Type VII of the form
$(\Omega_{2m}/\Omega_{2m}; \mathbb{O}_{48}/\mathbb{O}_{48})$,
Type VIII of the form
$(\Omega_{4m}/\Omega_{2m}; \mathbb{O}_{48}/\mathbb{T}_{24})$, and
Type IX of the form
$(\Omega_{2m}/\Omega_{2m}; \mathbb{I}_{120}/\mathbb{I}_{120})$.
Because $\mathbb{D}_2$ is a subgroup of both $\mathbb{T}_{24}$ and $\mathbb{O}_{48}$
while $\mathbb{I}_{120}$ contains a subgroup conjugate in $\SU_2$ to $\mathbb{D}_3$,
if $\Gamma$ is a group of any of the above types, then $\Gamma$ contains a subgroup conjugate to
$\mathbb{D}_r$ for some $r \geq 2$.  By Lemma \ref{lem:Finite2QuadraticDihedral}, the Taylor expansion of
$\Hilb_{\R[\C^2]^\Gamma|\R}(x)$ has coefficient $0$ or $1$ in degree $2$, so that by Lemma \ref{lem:Reduced3firstTaylor},
$\Hilb_{\R[\C^2]^\Gamma|\R}(x) \neq \Hilb_A^{\on}(x)$ for any generic, effective weight vector with nonzero weights.

It follows that $\Gamma$ cannot be any of the finite subgroups of $\U_2$, which completes the proof of Theorem \ref{mainthm}.

\bibliographystyle{amsplain}
\bibliography{HSImpossibility}

\providecommand{\bysame}{\leavevmode\hbox to3em{\hrulefill}\thinspace}
\providecommand{\MR}{\relax\ifhmode\unskip\space\fi MR }
\providecommand{\MRhref}[2]{%
  \href{http://www.ams.org/mathscinet-getitem?mr=#1}{#2}
}
\providecommand{\href}[2]{#2}
\begin{thebibliography}{10}

\bibitem{EncMathSciAlgeGeomI}
\emph{Algebraic geometry. {I}}, Encyclopaedia of Mathematical Sciences,
  vol.~23, Springer-Verlag, Berlin, 1994, Algebraic curves. Algebraic manifolds
  and schemes, A translation of {{\i}t Current problems in mathematics.
  Fundamental directions, Vol. 23 (Russian)}, Akad. Nauk SSSR, Vsesoyuz. Inst.
  Nauchn. i Tekhn. Inform., Moscow, 1988, Translation by D. Coray and V. N.
  Shokurov, Translation edited by I. R. Shafarevich.

\bibitem{ACG}
Judith~M. Arms, Richard~H. Cushman, and Mark~J. Gotay, \emph{A universal
  reduction procedure for {H}amiltonian group actions}, The geometry of
  {H}amiltonian systems ({B}erkeley, {CA}, 1989), Math. Sci. Res. Inst. Publ.,
  vol.~22, Springer, New York, 1991, pp.~33--51.

\bibitem{Gonc}
J.~Basto~Gon{\c{c}}alves, \emph{Reduction of {H}amiltonian systems with
  symmetry}, J. Differential Equations \textbf{94} (1991), no.~1, 95--111.

\bibitem{BrunsHerzog}
Winfried Bruns and J{\"u}rgen Herzog, \emph{Cohen-{M}acaulay rings}, Cambridge
  Studies in Advanced Mathematics, vol.~39, Cambridge University Press,
  Cambridge, 1993.

\bibitem{CoxeterBook}
H.~S.~M. Coxeter, \emph{Regular complex polytopes}, second ed., Cambridge
  University Press, Cambridge, 1991.

\bibitem{CrawleyBoevey}
William Crawley-Boevey, \emph{Normality of {M}arsden-{W}einstein reductions for
  representations of quivers}, Math. Ann. \textbf{325} (2003), no.~1, 55--79.

\bibitem{DerskenKemperBook}
Harm Derksen and Gregor Kemper, \emph{Computational invariant theory},
  Invariant Theory and Algebraic Transformation Groups, I, Springer-Verlag,
  Berlin, 2002, Encyclopaedia of Mathematical Sciences, 130.

\bibitem{DuVal}
Patrick Du~Val, \emph{Homographies, quaternions and rotations}, Oxford
  Mathematical Monographs, Clarendon Press, Oxford, 1964.

\bibitem{DufresneThesis}
Emilie~Sonia Dufresne, \emph{Separating invariants}, ProQuest LLC, Ann Arbor,
  MI, 2008, Thesis (Ph.D.)--Queen's University (Canada).

\bibitem{FarHerSea}
Carla Farsi, Hans-Christian Herbig, and Christopher Seaton, \emph{On orbifold
  criteria for symplectic toric quotients}, SIGMA Symmetry Integrability Geom.
  Methods Appl. \textbf{9} (2013), Paper 032, 18.

\bibitem{Gessel}
Ira~M. Gessel, \emph{Generating functions and generalized {D}edekind sums},
  Electron. J. Combin. \textbf{4} (1997), no.~2, Research Paper 11, approx. 17
  pp. (electronic), The Wilf Festschrift (Philadelphia, PA, 1996).

\bibitem{GuilSternSTPhysics}
Victor Guillemin and Shlomo Sternberg, \emph{Symplectic techniques in physics},
  Cambridge University Press, Cambridge, 1984.

\bibitem{Haefliger}
Andr{\'e} Haefliger, \emph{Groupo\"\i des d'holonomie et classifiants},
  Ast\'erisque (1984), no.~116, 70--97, Transversal structure of foliations
  (Toulouse, 1982).

\bibitem{HerbigIyengarPflaum}
Hans-Christian Herbig, Srikanth~B. Iyengar, and Markus~J. Pflaum, \emph{On the
  existence of star products on quotient spaces of linear {H}amiltonian torus
  actions}, Lett. Math. Phys. \textbf{89} (2009), no.~2, 101--113.

\bibitem{HerbigSchwarz}
Hans-Christian Herbig and Gerald~W. Schwarz, \emph{The {K}oszul complex of a
  moment map}, J. Symplectic Geom. \textbf{11} (2013), no.~3, 497--508.

\bibitem{HerbigSeaton}
Hans-Christian Herbig and Christopher Seaton, \emph{The {H}ilbert series of a
  linear symplectic circle quotient},  (2013), arXiv:1302.2662 [math.SG].

\bibitem{HochsterTori}
M.~Hochster, \emph{Rings of invariants of tori, {C}ohen-{M}acaulay rings
  generated by monomials, and polytopes}, Ann. of Math. (2) \textbf{96} (1972),
  318--337.

\bibitem{KempfNess}
George Kempf and Linda Ness, \emph{The length of vectors in representation
  spaces}, Algebraic geometry ({P}roc. {S}ummer {M}eeting, {U}niv.
  {C}openhagen, {C}openhagen, 1978), Lecture Notes in Math., vol. 732,
  Springer, Berlin, 1979, pp.~233--243.

\bibitem{KLMR}
A.~Kriegl, M.~Losik, P.~W. Michor, and A.~Rainer, \emph{Lifting mappings over
  invariants of finite groups}, Acta Math. Univ. Comenian. (N.S.) \textbf{77}
  (2008), no.~1, 93--122.

\bibitem{LMS}
Eugene Lerman, Richard Montgomery, and Reyer Sjamaar, \emph{Examples of
  singular reduction}, Symplectic geometry, London Math. Soc. Lecture Note
  Ser., vol. 192, Cambridge Univ. Press, Cambridge, 1993, pp.~127--155.

\bibitem{Mather}
John~N. Mather, \emph{Differentiable invariants}, Topology \textbf{16} (1977),
  no.~2, 145--155.

\bibitem{MatsumuraBook}
Hideyuki Matsumura, \emph{Commutative ring theory}, second ed., Cambridge
  Studies in Advanced Mathematics, vol.~8, Cambridge University Press,
  Cambridge, 1989, Translated from the Japanese by M. Reid.

\bibitem{Molien}
Th. Molien, \emph{\"{U}ber die {I}nvarianten der linearen
  {S}ubstitutionsgruppen}, Sitzungsber. der K\"{o}nigl. Preuss. Akad. d. Wiss.
  \textbf{zweiter Halbband} (1897), 1152--1156.

\bibitem{MumfordFogarty}
David Mumford and John Fogarty, \emph{Geometric invariant theory}, second ed.,
  Ergebnisse der Mathematik und ihrer Grenzgebiete [Results in Mathematics and
  Related Areas], vol.~34, Springer-Verlag, Berlin, 1982.

\bibitem{GWSdiff}
Gerald~W. Schwarz, \emph{Smooth functions invariant under the action of a
  compact {L}ie group}, Topology \textbf{14} (1975), 63--68.

\bibitem{GWSliftingHomotopies}
\bysame, \emph{Lifting smooth homotopies of orbit spaces}, Inst. Hautes
  \'Etudes Sci. Publ. Math. (1980), no.~51, 37--135.

\bibitem{GWSkempfNess}
\bysame, \emph{The topology of algebraic quotients}, Topological methods in
  algebraic transformation groups ({N}ew {B}runswick, {NJ}, 1988), Progr.
  Math., vol.~80, Birkh\"auser Boston, Boston, MA, 1989, pp.~135--151.

\bibitem{GWSlifting}
\bysame, \emph{Lifting differential operators from orbit spaces}, Ann. Sci.
  \'Ecole Norm. Sup. (4) \textbf{28} (1995), no.~3, 253--305.

\bibitem{SjamaarLerman}
Reyer Sjamaar and Eugene Lerman, \emph{Stratified symplectic spaces and
  reduction}, Ann. of Math. (2) \textbf{134} (1991), no.~2, 375--422.

\bibitem{SturmfelsBook}
Bernd Sturmfels, \emph{Algorithms in invariant theory}, Texts and Monographs in
  Symbolic Computation, Springer-Verlag, Vienna, 1993.

\bibitem{WehlauPopov}
David~L. Wehlau, \emph{A proof of the {P}opov conjecture for tori}, Proc. Amer.
  Math. Soc. \textbf{114} (1992), no.~3, 839--845.

\bibitem{WehlauPolynomial}
\bysame, \emph{When is a ring of torus invariants a polynomial ring?},
  Manuscripta Math. \textbf{82} (1994), no.~2, 161--170.

\end{thebibliography}

\end{document}